\documentclass{article}
\usepackage[utf8]{inputenc}
\usepackage[T1]{fontenc}
\usepackage[english]{babel}
\usepackage{blindtext}
\usepackage{graphicx}  
\usepackage{tikz}
\usetikzlibrary{arrows}
\usepackage{amsmath}
\usepackage{xfrac}
\usepackage{xspace}
\usepackage{stmaryrd}
\usepackage{mathrsfs}
\usepackage{amsmath, amssymb, amsthm, amsfonts, mathtools,amscd}
\usepackage{graphicx, setspace, subfig}
\usepackage{frontespizio}

\newtheorem{definizione}{Definition}[section]
\newtheorem{teorema}{Theorem}[section]
\newtheorem{proposizione}{Proposition}[section]
\newtheorem{corollario}{Corollary}[section]
\newtheorem{lemma}{Lemma}[section]
\newtheorem{osservazione}{Remark}[section]

\newenvironment{dimostrazione}{\medskip\noindent{\bf Proof}\enspace} {\hfill\newline\smallskip}

\newcommand{\ZZ}{\mathscr{Z}}
\newcommand{\WW}{\mathscr{W}}
\newcommand{\res}{\operatorname{res}}
\newcommand{\falfa}[1]{%
    ^{\mspace{2mu}\underline{\mspace{-.5mu}#1\mspace{-.5mu}}\mspace{2mu}}%
}
\title{On the volume of some Fano K-moduli spaces}
\author{Salvatore Tambasco \footnote{salvatore.tambasco01@universitadipavia.it}}
\date{}

\begin{document}

\maketitle

\begin{abstract}
   \noindent We compute the CM volume, that is the degree of the descended CM line bundle on the coarse moduli space in two cases: on the Fano $K$-moduli space of Quartic del Pezzo in any dimension, and
on the $K$-moduli space of the log Fano hyperplane arrangements of dimension one and two. Furthermore, we relate these volumes to the Weil-Petersson volumes by extending the notion of Weil-Petersson metric in the log case. 
    
\end{abstract}

\section{Introduction}

The aim of this work is to provide a new direction into the study of some geometric properties of $K$-moduli spaces, that is moduli spaces of $K$-stable varieties in any dimension. Namely, we would like to compute new kind of numbers (\emph{invariants}) from integration of certain tautological classes on moduli spaces of higher dimensional (stable) varieties $M$, generalizing the well-known curve case. When trying to do so in general there are many issues, e.g. we need compactifications of the $K$-moduli space, definitions of tautological classes, etc. moreover $K$-moduli spaces of Fano varieties are known to be proper and projective (\cite{LWX19}, \cite{LWX15}, \cite{Oda15}), and in some cases it is possible to describe these explicitly as  GIT quotients. Then, we can try to compute canonical volumes by integrating a \emph{natural} cohomology class, such as the first Chern class of the CM line bundle $\lambda_{\mathrm{CM}}$ on these compact moduli spaces.
Moreover such canonical volumes are \emph{Weil-Petersson volumes}. Indeed, in \cite{LWX15} it is proven, in the absolute case, that the degree of the CM line bundle is related to the Weil-Petersson metric. Here we extend the latter result to the case of moduli of pairs. 

The explicit examples of Fano $K$-moduli spaces considered in this work, are provided in \cite{OSS} and \cite{Fuji19b}. In these examples, the $K$-moduli spaces are GIT quotients. In the following, we briefly explain the machinery adopted for calculating the CM volume, that is the the degree of the CM line bundle descended on the coarse moduli space. Let $X$ be a normal variety endowed with an action of a complex reductive Lie group $G. $ Consider the GIT quotient $M \colon =  X //_{L} G$  with linearization $L$. Suppose that the semistable locus of $M,$ denoted by $X^{ss}$, and the stable locus of $M,$ denoted by $X^s,$ satisfy $X^{ss}=X^{s}$. The latter condition gives the surjectivity of the \emph{Kirwan map},

$$ \kappa \colon H_G^{\bullet}(X) \longrightarrow H^{\bullet}(M). $$
With that map the cohomology ring of  $M$, can be described via the equivariant cohomology ring of $X$. When the Kirwan map is surjective, the Jeffrey-Kirwan residue formula \cite{JK95}

$$ \kappa(\eta) e^{\omega_0}[M] = \frac{C^{G}}{\mathrm{Vol(T)}} \mathrm{Res}\left(\mathcal{D}^2(Y) \sum_{F \in \mathcal{F}} \int_{F} \frac{i_F^{*}(\eta e^{\bar{\omega}}(Y))}{e_F (Y)} [dY] \right), $$

\noindent is the best candidate for computing intersection product of classes in this case. In the above formula $\omega_0$, in the left hand side, denotes the symplectic form of the Marsdsen-Weinstein reduction coming from the symplectic picture of the GIT quotients  \cite[Chapter 5]{Kir84}. In the right hand side $C^G$ is a constant depending on the $G$-action on $X,$  $\bar{\omega}$ denotes the equivariant extension to the symplectic form on $X$ \cite[Chapter 9, section 9.3]{Jef19}. $\mathrm{Vol}(T),$ and $[dY]$ are the volume of the maximal torus $T$ of $G,$ and the measure on its Lie algebra $\mathfrak{t}$ induced by the restriction to $\mathfrak{t}$ of the fixed inner product on the Lie algebra of $G, \mathfrak{g}.$ The set $\mathcal{F}$ is the set of the connected components of the fixed point set of the maximal torus action, while $F$ denotes a single connected component. The map $i_F \colon F \hookrightarrow M$ is the inclusion, and $e_F$ is the equivariant Euler class of the normal bundle to $F$ in $M.$ \\

Suppose that the \emph{GIT-stability conditions} matches with the \emph{$K$-stability conditions.} That means, $K-$(semi)stable points are also (semi)stable points in the GIT sense for an explicit family of Fano varieties $X$. This can happen in concrete examples when, in particular, we are able to relate the CM line bundle with a GIT linearization. Then, the main strategy for computing the volume can be outlined as follows: 

\begin{enumerate} 
\item Identify who is the CM line bundle on X in term of the generators of the Picard group $\mathrm{Pic}(X)$, where X is the explicit space where the group $G$ acts as above. To do so we will take special families $\pi \colon \mathcal{X} \rightarrow \mathbb{P}^1 \subset X$ of Fano varieties on which we can compare the CM line bundle with the generators of the Picard group.  
\item Then $\lambda_{\mathrm{CM}}$ on $X$ will be of the form $L$, that is an explicit linearization on $X$. Then we can compute the volume of such $L$ descended on the coarse moduli space $M$, denoted by  $\tilde{L}$ via the Jeffrey-Kirwan non abelian localization theorem $$[\tilde{L}]^{\mathrm{dim}M} \cdot [M]= \kappa(c_1(\tilde{L})^{\mathrm{dim}M})) e^{\omega_0} [M] $$ to get the CM volume of $M$. 
\end{enumerate}
As an outcome of the above strategy we are able to prove the following in the absolute case:

\begin{teorema} The CM volume of the $K$-moduli space of quartic del Pezzo quartic of dimension $n,$ $M_{\mathrm{dP}^4},$ is given by

 $$Vol(M_{dP^4}, \lambda_{CM}) = \left( \frac{c}{m} \right)^{m-3} \frac{1}{2^{m-1}}\left( \sum_{k: 0< m-2k \leq m} \frac{(m-2k)^{m-3}}{k! (m-k)!} \right),$$ 
 where $m=n+3$ and $c =8(n+1)(n-1)^n - \sum_{i=1}^n \binom{n+1}{i} (n+1)^{n+1-i}n^i(i-1)2^{i+1}.$
\end{teorema}

The existence of KE metrics in the smooth case of del Pezzo quartic was studied by Arezzo, Ghigi and Pirola in \cite{AGP}. The explicit identification of the K-moduli with a GIT quotient was done by Mabuchi and Mukai in \cite{MabMuk} for $n \leq 2$ and for $n>2$ by Spotti and Sun in \cite{SS}. With this in mind, and the work of \cite{LWX15}, the volume of the quartic del Pezzo K-moduli space is a Weil-Petersson volume.
\\
In the case of pairs, we are able to prove the following:

\begin{teorema} The (log) CM volume of two dimensional log Fano weighted hyperplane arrangements $M_d = (\mathbb{P}^2)^m//_{\mathcal{O}(d_1, ..., d_m)} \mathrm{SL}(3),$ where $d_i \in (0,1) \cap \mathbb{Q},$ and $\sum_{i=1}^m d_i <3,$ is given by: 
  \begin{gather*}
    \mathrm{Vol}(M_d, \lambda_{CM})  = \\
    =
    % \frac{-1}{6(2m-8)!}
    -\sum_{f \in \mathcal{A}(d)} \frac{ (-1)^{m_2(f)} }{6(2m-8)!}
    % \left(
    \sum_{j=0}^{2m-8} \binom{2m-8}{j}\binom{m+m_2-6-j}{m_2 + m_3 -3} \xi_2^j \xi_3^{2m-8 -j} +\\
    % \left(\frac{\delta}{3} - \delta_3 (f) \right)^{j} \left(
    %   \delta_1(f) - \frac{\delta}{3} \right)^{2m-8-j} \right) \\
    - % \frac{-1}{6(2m-8)!}
    \sum_{f \in \mathcal{B}(d)} \frac{ (-1)^{m_2(f)} }{6(2m-8)!}
    % \left(
    \sum_{j=0}^{2m-8} \binom{2m-8}{j}\binom{m+m_2 -6-j}{m_1 + m_2 - 3}
    \xi_1^j \xi_2^{2m-8-j}.
    % \left( \frac{\delta}{3} - \delta_3(f) \right)^{j}
    % \left(\frac{\delta}{3} - \delta_2 (f) \right)^{2m-8-j} \right)
  \end{gather*}

  Where the terms $m(f),$ and $\xi_j$ are explained later in section 4.3.
  
\end{teorema}

\noindent And, generalising the notion of Weil-Petersson metric in the log case we are able to conclude that the above volume is a genuine Weil-Petersson volume, that is the integral of the top volume form associated to the Weil-Petersson metric.

We outline this work as follows: 
Section 2 is devoted to provide a description of the CM line bundle in the log case by stating the general definition and the main properties. Using the outstanding effort of Guenancia and Paun\cite{GuePa} in understanding the \emph{global Monge Ampere equation} in the conic K\"{a}hler-Einstein case, and the effort of Berman in \cite{Ber16}, and lately by Tian and Wang in \cite{TiaWa} on understanding the YTD conjecture in the simple normal crossing log case we are able to define the Weil-Peterson metric for the smooth locus of $K$-polystable families of log Fano pairs along the line of Fujiki and Schumacher in \cite{FS90}. Then, generalising the results given in \cite{LWX15} we are able to extend it to the smooth locus of the base of the families of K-polystable log Fano pairs, and with the same arguments of \cite{LWX15} and \cite{ADL19} we are also able to extend as a \emph{current} to the whole log $K$-moduli space. More technically, in Theorem \ref{teo3} we prove that for a family $f: (\mathcal{X}, \mathcal{D}) \rightarrow B$ of $K$-polystable log Fano varieties, there exists a continuous Deligne's Pairing metric $h_{\mathrm{DP}}$ on the log CM line bundle $\lambda_{\mathrm{CM}, \mathcal{D}}$ such that the Weil-Petersson metric $\omega_{\mathrm{WP}}$ extends as a positive current to the whole base $B.$ 

In Section 3, motivated by the work of Spotti and Sun of 2015 \cite{SS}, Odaka, Spotti and Sun of 2016 \cite{OSS}, we compute the CM volume of the $K$-moduli space of del Pezzo quartics in any dimension using the procedure described before. In Theorem \ref{firstexa} we express the CM volume of the $K$-moduli space of quartic del Pezzo, and by \cite{JK01} we are able to extend the result also in the even case (Remark \ref{evencase}).

Finally, in the third and last section we compute the log CM volume of the $K$-moduli space of log Fano hyperplane arrangements, with the procedure described above.

\section*{Acknowledgments}

I would like to express my most sincere gratitude towards my supervisors Prof. Cristiano Spotti and Prof. Alessandro Ghigi for their patient guidance and advice they have provided me. I wish to extend my thanks to Prof. Gergely Berczi for helpful conversations, and to the Mathematics Department of Aarhus University for hosting me for two years. I am grateful to the AUFF starting grant 24285, and to the Villum Fonden 0019098 for partially founding my research.

\section{The CM line bundle and the log Weil-Petersson metric}

We begin this section by recalling some basis on the CM line bundle also in the case of pairs. As it is customary in the literature of the \emph{area} the multiplicative and additive notations for line bundles and divisors will be interchanged.
Fix $f : \mathcal{X} \rightarrow B$ to be a proper flat morphism of scheme of finite type over $\mathbb{C}.$ Some fundamental properties of the chosen family may be needed in the CM line bundle, see e.g. \cite{CodPat}. However, for our applications, we will use only smooth families with simple singularities (double points or simple normal crossing).  Let $\mathcal{L}$ be a relatively ample on $\mathcal{X}$ and assume that $f$ has relative dimension $n \geq 1$ and the base $B$
is irreducible. In the following, we recall that a $\mathbb{Q}-$line bundle $L$  on a scheme $X$ is a formal symbol $L= M^{\otimes 1/m}$ where $M$ is a line bundle on $X$ and $m$ is a positive integer. %$ namely $\forall b \in B, \ (\mathcal{X}_b, \mathcal{L}_b)$  has constant dimension $n.$ 

\begin{definizione} \cite{ADL19} Let $\mathcal{D}_i, \forall i=1,2, ..,k $ be a closed subscheme of $\mathcal{X}$ such that $f|_{\mathcal{D}_i} : \mathcal{D}_i \rightarrow B$ has relative dimension $n-1,$ and $f|_{\mathcal{D}_i}$ is proper and flat. Given $d_i \in  [0,1] \cap \mathbb{Q}$ we define the log CM $\mathbb{Q}-$line bundle of the data $(f, \mathcal{D}:= \sum_{i=1}^k d_i \mathcal{D}_i), \mathcal{L})$ to be 

$$ \lambda_{CM, \mathcal{D}} := \lambda_{CM} - \frac{n \mathcal{L}_b^{n-1} \cdot \mathcal{D}_b}{(\mathcal{L}_b^n)} \lambda_{\mathrm{CH}} + (n+1) \lambda_{\mathrm{CH}, \mathcal{D}}, $$  
Where $\lambda_{CM}$ is the CM line bundle defined in \cite[Definition 1]{PT09} and $\lambda_{CH}$ is the Chow line bundle, defined as the leading order term of the Hilbert Mumford expansion \cite[Theorem 4]{MK}, namely $\lambda_{CH}= \lambda_{n+1},$ and $\lambda_{CH, \mathcal{D}} = \bigotimes_{i=1}^k \lambda_{CH, f|_{\mathcal{D}_i}, \lambda|_{\mathcal{D}_i}}.$
\end{definizione}

\begin{proposizione} \cite[Proposition 2.23]{ADL19} \label{pr} Let $f : (\mathcal{X}, \mathcal{D}) \rightarrow B$ be a $\mathbb{Q}$-Gorenstein flat family of $n$ dimensional pairs over a normal proper variety $B.$ Then for any $\mathcal{L}$ relatively $f-$ample line bundle we have

\begin{equation} \label{cmlog} c_1(\lambda_{CM, \mathcal{D}}) = n \frac{ (-K_{\mathcal{X}_b} - \mathcal{D}_b) \cdot \mathcal{L}_b^{n-1}}{(\mathcal{L}_b^n)} f_{*} c_1 ( \mathcal{L})^{n+1} - (n+1)f_{*}((-K_{\mathcal{X}/B} - \mathcal{D}) \cdot c_1(\mathcal{L})^n). 
\end{equation}

\end{proposizione}

When computing the intersection number of the log CM line bundle on a $\mathbb{Q}$-Gorenstein flat family on the moduli of pairs, it is important to choose a ample line bundle on such family. In the state of the art of these possible choices, mainly we can make two of such. These choices are presented in the following

\begin{corollario}\label{cpsgm} Let $f: (\mathcal{X}, \mathcal{D} := \sum_{i=1}^k d_i \mathcal{D}_i) \rightarrow B$ be a $\mathbb{Q}$-Gorenstein flat family over a normal proper variety $B.$  We have

\begin{enumerate}

\item \cite[Section 1.2, Equation 1.7.a]{CodPat} Given $\mathcal{L}= -K_{\mathcal{X}/B} - \mathcal{D}$, then $$ c_1(\lambda_{CM, \mathcal{D}}) = -f_{*}c_1(-K_{\mathcal{X}/B} - \mathcal{D})^{n+1}.$$
\item \cite[Theorem 2.7]{SGM18} Given $\mathcal{L}= -K_{\mathcal{X}/B},$ and $\mathcal{D}_{\mathcal{X}_b} \in | -K_{\mathcal{X}_b} |, \ \forall b \in B$, then
\begin{align*} c_1 (\lambda_{CM, \mathcal{D}}) &= f_{*} (-K_{\mathcal{X}/B})^{n} \cdot (-c_1 (-K_{\mathcal{X}/B}) \\ &+ \sum_{i=1}^k d_i((n+1) \mathcal{D}_i - n c_1 (-K_{\mathcal{X}/B})) .
\end{align*}  
\end{enumerate}
\end{corollario}

\begin{osservazione} In the above, if we put $\mathcal{D}=0$ then we recover the absolute case studied in \cite{PT09}, \cite{FR06}. 

\end{osservazione}

\subsection{On the log Weil-Petersson metric.}

Consider a $\mathbb{Q}$-Gorenstein smoothable family $f: (\mathcal{X}, \mathcal{D}) \rightarrow B$ over a complex space $B.$ Suppose $f$ is of pure dimension $n,$ $(\mathcal{X}, \mathcal{D})$ is a log Fano Kawamata log terminal (klt) Pair. We assume, henceforth, that the divisor $\mathcal{D}= \sum_{k=1}^m(1-\beta_k)[s_k=0]$ is fiberwise simple normal crossing. In this setting we have a K\"{a}hler metric $\omega$ on $X \setminus \bigcup_k D_k$ which is quasi isometric, to the model cone metric with cone angles $2 \pi \beta_k$ along $[z_k=0]$

$$ \omega_{\mathrm{cone}} = \sum_{k=1}^m \frac{1}{|z_k|^{2(1 - \beta_k)}}\sqrt{-1}dz_k \wedge d\bar{z}_k + \sum_{k=m+1}^n \sqrt{-1} dz_k \wedge d \bar{z}_k. $$
The above is nowadays a standard result in the literature, see \cite{GuePa}, \cite{Gue12}, \cite{Gue13}.
One might ask whether one can find a K\"{a}hler-Einstein metric (KE) $\omega$ on $X \setminus \mathrm{supp}(\mathcal{D})$ having conic singularities along $\mathcal{D}.$ This metric will be referred henceforth as conic K\"{a}hler metric (cKE). In \cite[Theorem A]{GuePa} it is proven that the answer to that question is positive in the K-stable situation and it is sufficient to produce a weak solution $\omega_{\varphi} = \omega + \sqrt{-1} \partial \bar{\partial} \varphi$ of the so called \emph{global Monge-Ampere} equation

$$ \omega^n = \frac{e^{\varphi} dV}{\prod_k |s_k|^{2(1-\beta_k)}},$$

where $dV$ is a smooth volume form. Moreover it is proven \cite[Theorem B]{GuePa} that any solution $\varphi$ of the above equation is bounded and plurisubharmonic, more precisely $\varphi \in \mathcal{C}^{2,\alpha, \beta},$ where $C^{2, \alpha, \beta}$   is a \emph{conic H\"{o}lder space} defined in \cite[sec. 7.1 Definition 1]{GuePa}.

The log K-stability was introduced in \cite{Li11}. It is proved in \cite{Ber16} that if a log Fano pair $(X,D)$ admits a cKE metric, then it is log K-polystable as defined in \cite{Li11}. The converse is also true and it has been recently proven by \cite{TiaWa}. Therefore we can assume that the above family $f: (\mathcal{X}, \mathcal{D}) \rightarrow B$ has log K-polystable fibers. Namely, $\forall b \in B, \ (\mathcal{X}_b, \mathcal{D}_b)$ is a log K-polystable Fano pair with cKE metric $\omega_b = \omega_0 + \sqrt{-1}\partial \bar{\partial} \varphi,$ where $\varphi \in \mathcal{C}^{2, \alpha, \beta},$ is a solution of the global Monge Ampere equation, and $\omega_0 \in c_1 (-K_{\mathcal{X}/B} - \mathcal{D}).$ Let $A$ be the subset of $B$ parametrizing the singularities of $B,$ denote its complement, i.e. the \emph{smooth locus} by $B^0$ and by $f^0 : \mathcal{X}^0 \rightarrow B^0$ the corresponding smooth family. For all $b \in B^0$ the corresponding cKE metric varies in $\mathcal{X}_b^0.$ At every $b \in B^0$ we obtain a Hermitian metric $\{\omega_b ^n\}$ on $-K_{\mathcal{X}^0 / B^0} - \mathcal{D}.$ At every $b \in B$, we can produce a map

$$ \mathcal{E} : C^{2, \alpha , \beta} \times [0,1] \rightarrow C^{\alpha, \beta} $$defined as

$$ \mathcal{E}( \varphi, \epsilon) = \log \frac{\omega_{b, \varphi_{\epsilon}}}{\Omega_{b,\epsilon}} - ( F_{\epsilon} + \phi_{\epsilon}). $$
Where $\phi_{\epsilon}$ is a smooth approximation of the solution $\varphi$ and $$F_{\epsilon}= - \log \left( \frac{\prod_k ( \epsilon^2 + |s_k|^2)^{1- \beta_k}}{\Omega_{b, \epsilon}} \right).$$
The derivative of $\mathcal{E}$ in $\varphi$ leads to the \emph{Laplacian operator}, which by \cite{GuePa}, and \cite{Don12},  has bounded estimates. Moreover, also in \cite{GuePa} it is proven that $\omega_{b,\varphi_{\epsilon}}$ converges smoothly to $\omega_{b, \varphi}$ outside the support of the divisor and along $\mathcal{D}$ it has only conic singularities. We denote by $\omega_{\mathcal{X}^0}$ its Chern curvature, namely

$$ \omega_{\mathcal{X}^0} = -\frac{\sqrt{-1}}{2 \pi}  \partial \bar{\partial} \log \{ \omega_b^n\}. $$

Because of \cite[Lemma 4,section 7.1.2.]{GuePa} we still have a control on the derivatives of such metrics, namely these latters belong to the class $C^{\alpha, \beta}.$ In this setting, we define as in \cite{FS90}, and \cite[Theorem 4.1]{LWX15}.
  
\begin{definizione} \label{logwp} In the above setting, we define the \emph{log Weil-Petersson} metric as the fiber integral

\begin{equation} \label{eq1} \omega_{\mathrm{WP}}^0 := - \int_{\mathcal{X}^0 / B^0} \omega_{\mathcal{X}^0}^{n+1}.\end{equation}

\end{definizione}

In \cite{FS90} it is proven in the non log case that there exists a Quillen metric $h_{\mathrm{QM}}$ on $\lambda_{\mathrm{CM}_{|B^0}}$ such that its Chern curvature is the Weil-Petersson metric. In \cite{Del87} Deligne showed how to calculate such Quillen metric, via the so called \emph{metrized Deligne Pairing}.
In the next result we collect two results that will be useful for later discussions.

\begin{teorema} \label{teo1} Let $g: \mathcal{X} \rightarrow B$ be a flat projective morphism of integral schemes of finite type, of relative dimension $n.$ Let $(\mathcal{L}_0, h_0), ...,( \mathcal{L}_n,h_n)$ be a $(n+1)-$tuple of hermitian line bundles on $\mathcal{L},$ denote the Deligne pairing associated to that tuples as the symbol $< \mathcal{L}_0 , ..., \mathcal{L}_n>.$ We have

\begin{enumerate}
\item (\cite[Theorem 2.1]{Mor99}) Suppose $h_i$ are smooth metrics $\forall i \in \{0,1,...,n\}.$ Then the Deligne's metric $h_{DP}$ is continuous on $<\mathcal{L}_0, ...,  \mathcal{L}_n>.$
\item (\cite[Proposition 8.5]{Del87}) The following Chern-Weil curvature formula holds for Deligne's metric 
$$ c_1 (<\mathcal{L}_0, ..., \mathcal{L}_n>) = \int_{\mathcal{X}/B} c_1(\mathcal{L}_0) \wedge ... \wedge c_1(\mathcal{L}_n).$$
\end{enumerate}  

\end{teorema}

Henceforth, we will denote the Deligne's Pairing of a repeated line bundle $\mathcal{L}$ as

$$ \mathcal{L}^{<n+1>} = \underbrace{<\mathcal{L}, ..., \mathcal{L}>}_{\text{$(n+1)-$times}}.$$

The previously defined log CM line bundle can be interpreted with the Deligne Pairing.  Namely, along the line of \cite[Section 2.3]{Ber16}, \cite[Theorem 4.9]{LWX15}, we can define

$$\label{teo2} \lambda_{\mathrm{CM}, \mathcal{D}} = - \left(-K_{\mathcal{X}/\mathcal{B}} - \mathcal{D} \right)^{<n+1>} $$
Because of  Corollary \ref{cpsgm} of previous section,for a $\mathbb{Q}$-Gorenstein family of log Fano pair $f: (\mathcal{X}, \mathcal{D}) \rightarrow B$ the first Chern class of the log CM line bundle is given by $c_1 (\lambda_{\mathrm{CM}, \mathcal{D}})= - f_{*}c_1(-K_{\mathcal{X}/B} - \mathcal{D})^{n+1}.$ It can be rephrased in the following way

$$ c_1 (\lambda_{\mathrm{CM}, \mathcal{D}})= - \int_{\mathcal{X}/B} \underbrace{c_1(-K_{\mathcal{X}/B} - \mathcal{D}) \wedge ... \wedge c_1(-K_{\mathcal{X}/B} - \mathcal{D})}_{\text{$(n+1)-$times}}. $$

Thus by point 2 in \ref{teo1} above, we have

$$ c_1 (\lambda_{\mathrm{CM}, \mathcal{D}})= c_1  \left((-K_{\mathcal{X}/B} - \mathcal{D})^{<n+1>}\right)$$

Now we are ready to prove a new statement, that extends the notion of Weil-Petersson metric in the case of  log Fano pairs.

\begin{teorema} \label{teo3} Let $f: (\mathcal{X}, \mathcal{D}) \rightarrow B$ be a $\mathbb{Q}-$Gorenstein smoothable family of pure dimension $n,$ where $(\mathcal{X}, \mathcal{D})$ is a log klt pair, and the divisor $\mathcal{D}$ is fiberwise simple normal crossing. Then there exists a continuous metric $h_{\mathrm{DP}}$ on $\lambda_{\mathrm{CM}, \mathcal{D}}$ such that 
\begin{enumerate}
\item $\omega_{\mathrm{WP}} := \frac{\sqrt{-1}}{2 \pi} \partial \bar{\partial} \log h_{\mathrm{DP}}$ is a positive current.
\item $\omega_{\mathrm{WP}_|B^0} = \omega_{WP}^0.$
\end{enumerate}

\end{teorema}

\begin{dimostrazione} It is proved in \cite{TiaWa} that the partial $C^0-$estimates naturally extends in the log case for symple normal crossing pairs. Therefore, as explained at the beginning of  \cite[Section 3]{LWX15} by embedding the total space into a suitable projective space, we can define two sets of metrics along fibers $\{\omega_b^n \},$ and $\{\tilde{\omega}_b^n\}.$
By assertion 1 of \ref{teo1} both sets, outside the support of the divisor, define Deligne's metrics on $\lambda_{\mathrm{CM}, \mathcal{D}},$ which we call $h_{\mathrm{DP}}$ and $\tilde{h}_{\mathrm{DP}}$ respectively.  Since, the Chern curvature depends fiberwise on the $\partial \bar{\partial}$  derivatives of $\omega_b^n,$ for all $b \in B$, then by \cite[Lemma 4 in section 7.1.2.]{GuePa}  we still have a control on the derivatives of such metrics. Therefore we can consider the Deligne's metrics $h_{\mathrm{DP}},$ and $\tilde{h}_{\mathrm{DP}}$  to be well defined also along the Divisors.

By the \emph{change of metric} formula of Deligne's pairing \cite{LWX15}, \cite{Del87}, \cite{Mor99} we have 

$$ h_{\mathrm{DP}}= \widetilde{h_{\mathrm{DP}}} \cdot e^{-\Phi}, $$

Where 

$$ \Phi = -\frac{1}{(n+1)!} \sum_{i=1}^n \int_{\mathcal{X}_b} \tilde{\varphi} \tilde{\omega}_b^{i} \wedge \omega_b^{n-i}. $$

As it is proved in \cite[Proposition 2.14]{SSY15}, $\Phi$ is continuous and uniformly bounded. Notice that over each fiber $(\mathcal{X}_b,\mathcal{D}_b)$ we have $\omega_b = \tilde{\omega}_b + \frac{\sqrt{-1}}{2 \pi} \partial \bar{\partial} \tilde{\varphi}.$ By looking at the curvature data we find

$$ \omega_{\mathcal{X}}= \frac{\sqrt{-1}}{2 \pi} \partial \bar{\partial} \log \{\omega_b^n\} = \tilde{\omega}_{\mathcal{X}} + \frac{\sqrt{-1}}{2 \pi} \partial \bar{\partial} \tilde{\varphi}.$$

In the smooth part we have that the following identity holds:

$$ \omega_{\mathcal{X}^0}^{n+1}= \tilde{\omega}^{n+1} + \frac{\sqrt{-1}}{2 \pi} \partial \bar{\partial} \Phi,$$ 

where we denoted by $\tilde{\omega}$ the curvature data of $\{\tilde{\omega}_b^n\}.$ Notice that both $\omega_{\mathcal{X}^0}$ and $\tilde{\omega}$ are smooth $(1,1)$-forms over $\mathcal{X}^0,$ so we can take fiber integrals 

$$ \int_{\mathcal{X}^0 / B^0} \omega_{\mathcal{X}^0}^{n+1}= \int_{\mathcal{X}^0/B^0} \tilde{\omega}^{n+1} + \frac{\sqrt{-1}}{2 \pi} \partial \bar{\partial} \Phi,$$

we see that the above identity is nothing but

$$ \omega_{\mathrm{WP}}^0 = - \frac{\sqrt{-1}}{2 \pi} \partial \bar{\partial} \log \tilde{h}_{\mathrm{DP}} + \frac{\sqrt{-1}}{2 \pi} \partial \bar{\partial} \Phi  = - \frac{\sqrt{-1}}{2 \pi} \partial \bar{\partial} \log h_{\mathrm{DP}|B^0}. $$

The claim follows from regularity results in pluripotential theory, see \cite[Lemma 4.13]{LWX15}. \hfill $\Box$

\end{dimostrazione}
It is well known that the weight of the CM line bundle coincides with the Donaldson-Futaki invariant. In particular it is zero on $K$-polystable points and thus it descends as a $\mathbb{Q}$-line bundle on the $K$-proper moduli space. Then, there is a well defined continuous Deligne's pairing metric whose curvature is a positive current as in \ref{teo3}.  In conclusion, to compute the log Weil-Petersson volume we can simply compute the degree of the descended CM line bundle.

\section{The CM Volume of quartic del Pezzo moduli space}

We begin this section with a brief description of the intersection of two quadrics and then we describe the moduli space used for computing the CM volume of the latter in any dimension. 

\subsection*{A brief description of the intersection of two quadrics}

Let $X= Q_1 \cap Q_2$ be the intersection of two quadrics $Q_1$ and $Q_2,$ in $\mathbb{P}(V),$ where $V$ is a complex vector space of dimension $n+1.$ Consider a pencil of the quadrics $Q_1$ and $Q_2,$ that is 
 $\Phi = Span \{ \phi_{\lambda}\}_{\lambda \in \mathbb{P}^1}.$Where every $\phi_{\lambda}$ denotes the bilinear form associated to the quadric $Q_{\lambda} = Q_1 + \lambda \cdot Q_2$ in $\mathbb{P}^n$. A pencil of quadrics $\Phi$ is said to be smooth, or non singular, if  $Q_{\lambda}$ is non singular or have at most simple cones as singularities (that is an ordinary double point) $\forall \lambda \in \mathbb{P}^1$, \cite[Chapter 1]{Reid}. For the $\lambda'$s for which $Q_{\lambda}$ is degenerate, if $e_{\lambda}$ is a basis for $ker(\phi_{\lambda}),$ then $\phi_{\mu}(e_{\lambda}) \neq 0$ for all $\lambda \neq \mu.$ If $e \in V$, we define the orthogonal complement with respect to $\Phi$ as follows 

$$ e^{\perp} = \cap_{\lambda \in \mathbb{P}^1} e^{\perp_{\phi_{\lambda}}}= e^{\perp_{\phi_1}} \cap e^{\perp_{\phi_2}} $$

Note that, for most of $e$, the above will be of codimension two. If the codimension of $e^{\perp}$ is less than two, then $e \in ker(\phi_{\lambda})$ for some $\lambda \in \mathbb{P}^1.$ Hence $X$ is nonsingular and of codimension two in $\mathbb{P}(V)$  if and only if for all $x \in X$, $e \in V$ representing $x$, then $e$ is not in the kernel of $\phi_{\lambda}$ for any $\lambda \in \mathbb{P}^1.$ Smoothness of $X$ can be characterized through smoothness of $\Phi$ as the following proposition shows

\begin{proposizione} (\cite{Reid}) Let $X$ be as above, then the following conditions are equivalent 

\begin{enumerate} \item $X$ is smooth and of codimension 2,
\item $\Phi$ is non singular,
\item The discriminant $det(\phi_{\lambda}) \neq 0$ is a polynomial in $\lambda$ of degree the dimension of $V$ and has $n+1$ distinct roots,
\item there exists a unique basis of $V$ $\{e_i\}_{i=1}^{dimV}$ such that
\begin{equation}
   \begin{cases}
   \phi_1 ( \sum_i x_i e_i)= \sum_i x_i ^2 \\ \phi_2( \sum_i x_i e_i)= \sum_i \lambda_i x_i ^2, \ \lambda_i \neq \lambda_j, \ i \neq j
   \end{cases}
\end{equation}
\end{enumerate}

\end{proposizione}

\subsection{The moduli space of the intersection of Quadrics}

Denote by $M_{\Phi}$ the moduli space of pencils of quadrics. This latter moduli space arises as a GIT quotient as shown by Avritzer and Lange \cite{AL} where the stability conditions are described in the following

\begin{teorema}(\cite{AL}) Consider $Gr(2, N+1),$ $N = \frac{1}{2}n(n+3)$ , as the space of pencils of quadrics in $\mathbb{P}^n$, where $n$ is the dimension of the intersection of two quadrics . Then,
\begin{itemize} \item[a.] A pencil $\Phi \in Gr(2,N+1)$ is semistable if and only if its discriminant admits no roots of multiplicity $> n/2$.
\item[b.] A pencil $\Phi \in Gr(2,N+1)$ is stable if and only if its discriminant admits no multiple roots.
\end{itemize}
\end{teorema}

In the previous section we saw that if the discriminant admits no multiple root, then the pencil is smooth, smoothness of pencils is equivalent to the smoothness of the intersection of two quadrics. With this in mind, we can easily see that stable intersection of quadrics are given by all simultaneously diagonalizable quadrics, therefore we can describe the moduli space of quartic del Pezzo's (intersection of two quadrics) by the roots $\lambda \in \mathbb{P}^1.$ 
\\
In \cite[Theorem 1.1]{SS} it is proven that the K-moduli space of quartic del Pezzo in any dimension can be explicitly described as follows

$$ \overline{\mathcal{M}_{dP^4}^K} \simeq Gr(2, \mathrm{Sym}^2(\mathbb{C}^{n+3})) // SL(n+3) $$

By Avritzer and Lange in \cite[Section 4]{AL}, the above GIT quotients can be identified with 

$$ M_{dP^4} = S^m \mathbb{P}^1 //_{\mathcal{L}} \mathrm{SL}(2). $$

Where $S^m \mathbb{P}^1$ denotes the symmetric power of degree $m=n+3$ of the projective line. The linearization $\mathcal{L}$ is given by $\mathcal{O}(1).$ $M_{dP^4}$ is a genuine $K$-moduli space as shown in \cite{SS}, Theorem 1.1.

\subsection{The intersection number of the CM line bundle with a smooth curve in the moduli of quartic del Pezzo's}

We begin this section computing the intersection number of the CM line bundle with a smooth curve.

\begin{proposizione}Let $\gamma \colon \mathcal{X} \rightarrow C$ be a flat and proper family of intersection of quadrics of relative dimension $n,$ over a smooth curve $C$ of genus $g.$ Assume that $-K_{\mathcal{X} /C}$ is $\gamma$-ample. Then, the complete curve of the intersection of the CM line bundle with the above family is given by

\begin{equation}\label{CM1} deg(\lambda_{CM}(\mathcal{X} \rightarrow C))=8(n+1)(n-1)^n(1-g) - c_1 (\mathcal{X})^{n+1}
\end{equation}

\end{proposizione}

\begin{dimostrazione} We apply Corollary \ref{cpsgm} in the case $\mathcal{D}=0$ to get 

$$ c_1(\lambda_{CM}(\mathcal{X} \rightarrow C))= - \gamma_{*} c_1 (-K_{\mathcal{X}/C})^{n+1}.$$

By the definition of relative canonical bundle we find

\begin{align*}
c_1(\lambda_{CM}(\mathcal{X} \rightarrow C)) &= -\gamma_{*} c_1 (K_{\mathcal{X}} \otimes \gamma^{*} K_C^{-1})^{n+1}  \\ &= - \gamma_{*} \left( \sum_{i=0}^{n+1} \binom{n+1}{i} c_1(K_{\mathcal{X}})^{n+1-i} \cdot c_1 (\gamma^{*} (-K_C))^i \right). 
\end{align*}

Since $C$ is a curve, then higher powers of $c_1(K_C)$ must vanish. Therefore,

\begin{equation} \label{fiodena} c_1(\lambda_{CM}(\mathcal{X} \rightarrow C))= - \gamma_{*} (c_1(K_{\mathcal{X}})^{n+1} - (n+1)c_1(K_{\mathcal{X}})^n \cdot c_1 (\gamma^{*}K_C)). 
\end{equation}

By the Riemann-Roch theorem for curves we find $c_1(K_C)=(2g-2)H,$ where $H$ is the generator of $H^2(C, \mathbb{Z})$. 
We shall now prove that $\gamma_{*}c_1 (-K_{\mathcal{X}})^n=4(n-1)^n.$ By substituting these two results in \ref{fiodena} and using the following projection formula (see \cite[page 14]{BHV}
$$\gamma_{*} (c_1(K_{\mathcal{X}})^n \cdot c_1(\gamma^{*} K_C)) = \gamma_{*}(c_1(K_{\mathcal{X}})^n)\cdot c_1(K_C) $$
the claim follows.
Fiberwise we have that $K_{\mathcal{X}}=K_{Q_0 \cap Q_2}.$ Then,

\begin{align*}
K_{Q_0} &= \left( K_{\mathbb{P}^{n+2}} + \mathcal{O}_{\mathbb{P}^{n+2}}(Q_0)\right)_{\vert Q_0} 
\\ &= \left( \mathcal{O}_{\mathbb{P}^{n+2}}(-n-1) \right)_{\vert Q_0} = \mathcal{O}_{Q_0}(-n-1). \end{align*}

\begin{align*}
K_{Q_0 \cap Q_2} &= \left(K_{Q_0} + \mathcal{O}_{Q_0} (Q_2) \right)_{\vert Q_0 \cap Q_2} \\ &= \mathcal{O}_{Q_0 \cap Q_2} (1-n).
\end{align*}

Then

\begin{align*} \gamma_{*}c_1 (-K_{\mathcal{X}})^n&= \int_{\mathcal{X}_c} c_1 ((n-1)\mathcal{O}_{\mathcal{X}}(1))^n
\\ &= (n-1)^n \int_{\mathcal{X}_c}c_1 (\mathcal{O}_{\mathcal{X}}(1))^n \\ &= (n-1)^n H^n \cdot Q_0 \cdot Q_2 = 4(n-1)^n. \ \ \ \ \ \ \ \ \  \hfill \Box
    \end{align*}

\end{dimostrazione}

The family chosen for Equation \ref{CM1} can be specified by fixing some smooth quadric $Q_0$, and letting $Q_1, Q_2$  varying in $\mathbb{P}^1$ through the following Lefschetz pencil

\begin{equation} \label{fam}
    Q_0 \cap (s Q_1 + t  Q_2) =0, \ \ (s:t) \in \mathbb{P}^1. 
\end{equation}  

This is by construction a smooth family of intersection of quadrics in $\mathbb{P}^{n+2}$ that has only nodal singularities. By definition of Lefschetz pencil \cite{Voi02}, the base locus $B$ is smooth and of codimension two in $Q_0.$ The total space of the family $\gamma$ can be specified by considering the blow up in the base locus of the quadric $Q_0,$ i.e.

$$ \mathcal{X} = Bl_{B}(Q_0). $$

Because of this construction the family $\gamma$ becomes a fibration over $\mathbb{P}^1,$ i.e.

$$ \gamma \colon \mathcal{X} \rightarrow \mathbb{P}^1. $$

Hence, with respect to the above family, equation \ref{CM1} becomes

\begin{equation} c_1(\lambda_{CM}(\mathcal{X} \rightarrow C))=8(n+1)(n-1)^n - c_1 (\mathcal{X})^{n+1}
\end{equation}

Hence, in order to calculate the above intersection number it sufficies to calculate the $(n+1)-$power of the first Chern class of $\mathcal{X}.$ Denoting by $p$ the blow up map, and using the fact that $\mathcal{X}$ is Fano we have

$$ c_1 (\mathcal{X})^{n+1} = - K_{\mathcal{X}} ^{n+1} = ( (n+1)H - nE)^{n+1}. $$

In order to calculate the contributions of the intersection products of the above we need the following

\begin{lemma} $$ c_1 (N_{B | Q_0}) = 4 c_1 (\mathcal{O}_B (1))$$
$$ c_2 (N_{B | Q_0}) = 4 c_1 (\mathcal{O}_B (1))^2 $$
\end{lemma}

\begin{dimostrazione} Consider the following chain of inclusions

$$ B \hookrightarrow Q_0 \cap Q_1 \hookrightarrow Q_0, $$

for which the short exact sequence of normal bundles follows 

$$ 0 \rightarrow N_{B | Q_0 \cap Q_2} \rightarrow N_{B | Q_0} \rightarrow N_{Q_0 \cap Q_2 | Q_0} \rightarrow 0. $$

The rest follows easily from the functorial properties of the total Chern class.   \hfil $\Box$
\end{dimostrazione}

By the general theory of blow ups, it follows that 

\begin{equation} c_1 (E)^2 = - p^{*} c_1 (N_{B | Q_0}) \cdot c_1 (E) - p^{*}c_2 (N_{B|Q_0}). \end{equation}

Set $x= c_1 (E)$ and $h = p^{*}c_1(\mathcal{O}_B(1))$. Then by the above Lemma we have $ p^{*}c_1 (N_{B|Q_0}) = 4h$ and $ p^{*}c_2 (N_{B|Q_0}) = 4h^2,$  so

$$ x^2 = -4hx - 4h^2. $$

\begin{lemma} $$x^k = (-1)^{k-1} (k 2^{k-1} h^{k-1} x + (k-1)2^k h^k), \ \ \forall k \in \mathbb{N}_{+}$$
\end{lemma}

\begin{dimostrazione} Set $x^k = A_k h^{k-1}x + B_k h^k,$ where $A_k = (-1)^{k-1}k 2^{k-1}$ and $B_k=(-1)^{k-1} (k-1)2^k$

For $k=2$ the above is true. Therefore, we can inductively assume that the assertion holds for $x^k,$ and we prove it for the next step

\begin{align*} 
x^{k+1} &=x^k x = A_k h^{k-1}x^2 + B_k h^k x \\ 
 &= A_k h^{k-1}(-4hx-4h^2) + B_k h^k x  \\
 &=-4A_k h^k x - 4 A_k h^{k+1} + B_k h^k x \\ 
 &= (B_k-4A_k)h^k x - 4A_k h^{k+1}. 
\end{align*} 

Therefore, the assertion holds if and only if 

\begin{equation} A_{k+1} = B_k - 4 A_k \ \ \ \ \ \ \ \ B_{k+1}=-4 A_k \end{equation}

Using the above definition of $A_k$ and $B_k,$ equations (5) hold. Hence, the assertion follows. \hfil $\Box$

\end{dimostrazione}

With this in mind we consider the binomial expansion

%\begin{align*}  c_1 (\mathcal{X})^{n+1} &= - K_{\chi} ^{n+1} = ( (n+1)H - nE)^{n+1} \\
%&= (n+1)^{n+1} H^{n+1} -n(n+1)^{n}H^{n}E + \cdots  \\ &+ {{n+1-i} \choose {i}} %(n+1)^{n+1}H^{n+1-i}(-nE)^{i} 
%\\ &+ \cdots - E^{n+1}.
%\end{align*}
\begin{align*}
  c_1 (\mathcal{X})^{n+1} &= - K_{\chi} ^{n+1} = ( (n+1)H - nE)^{n+1} \\
  &= (n+1)^{n+1} H^{n+1} -n(n+1)^{n}H^{n}E + \cdots \\
  &\quad + \binom{n+1}{i} (n+1)^{n+1}H^{n+1-i}(-nE)^{i} \\
  &\quad + \cdots - nE^{n+1}.
\end{align*}

Now, we calculate each contribution of the various intersection products. We note immediately that $H^{n+1}=1,$ and 
as a consequence of the moving lemma we see that $H^n E =0.$
By using the previous two lemmas we can compute the general term

\begin{align*} 
H^{n+1-i}E^i &= \int_{H^{n+1-i}E} c_1 (E)^{i-1} \\
&=  \int_{H^{n+1-i}E} (-1)^{i-2} [(i-1)2^{i-2}h^{i-2}x + (i-2)2^{i-1}h^{i-1}] \\
&= (-1)^{i-2} \int_{H^{n+1-i}E} (i-1)2^{i-2} p^{*}c_1(\mathcal{O}_B(1))^{i-2} \smile c_1(E) \\
&+ (-1)^{i-2} \int_{H^{n+1-i}E} (i-2)2^{i-1}p^{*}c_1(\mathcal{O}_B (1))^{i-1} \smile 1 \\
&= (-1)^{i-1} (i-1) 2^{i-2}  \int_{H^{n+1-i}B} c_1 (\mathcal{O}_B (1))^{i-2} \\
&= (-1)^{i-1}(i-1) 2^{i-2} H^{n+1-i} \cdot Q_0 \cdot Q_1 \cdot Q_2 \cdot H^{i-2} \\
&= (-1)^{i-1}(i-1)2^{i-2} \cdot 8 = (-1)^{i-1} (i-1) 2^{i+1}.
\end{align*}

Hence the desired intersection number is given by

%\begin{align} \label{inter}deg(\lambda_{CM}) \cdot (\mathcal{X} \rightarrow C) &=8(n+1)(n-1)^n  \nonumber \\ &- \sum_{i=1}^n {{n+1} \choose {i}} (n+1)^{n+1-i}(i-1)2^{i+1} \end{align}
\begin{align} \label{inter}
  \text{deg}(\lambda_{CM}) \cdot (\mathcal{X} \rightarrow C) &= 8(n+1)(n-1)^n \nonumber \\
  &- \sum_{i=1}^n \binom{n+1}{i} (n+1)^{n+1-i}(i-1)2^{i+1}
\end{align}
\subsection{The CM volume of the moduli space of quartic del Pezzo}

We start this section by recalling the definition of del Pezzo variety.
\begin{definizione}(\cite[p. 53-55]{Shaf})  A del Pezzo variety is a pair $(X,H)$ consisting of an irreducible projective variety $X$ of dimension $n$ and ample Cartier divisor $H$ on $X$ such that.
\begin{itemize}
    \item $X$ has only $\mathbb{Q}-$Gorenstein singularities.
    \item $-K_X = (n-1)H$
    \item $H^q (X, \mathcal{O}_X (tH)) = 0, \ \forall q,t, 0 < q < n$
\end{itemize}
    The degree $d$ of $X$ is defined as  $d = H^n$ 
\end{definizione}
A quartic del Pezzo $(X,H)$ is a del Pezzo variety with degree $d=4,$ hence by \cite[Theorem 3.2.5, assertion (iv)]{Shaf} is a complete intersection of two quadrics. 

\begin{teorema} \label{firstexa} Let $M_{dP^4}$ be the K-moduli space of quartic del Pezzo's of even dimension $n$. The CM volume of $M_{dP^4}$ is given by

\begin{equation} \label{Voluquarti} Vol(M_{dP^4}, \lambda_{CM}) = \left( \frac{c}{m} \right)^{m-3} \frac{1}{2^{m-1}}\left( \sum_{k: 0< m-2k \leq m} \frac{(m-2k)^{m-3}}{k! (m-k)!} \right).  \end{equation}
Where $c=deg(\lambda_{CM}) \cdot (\mathcal{X} \rightarrow C)$ is the intersection number computed in \ref{inter} and $m=n+3$.
\end{teorema}

\begin{dimostrazione} Observe that , since $S^m \mathbb{P}^1$ can be identified with a $m$-dimensional projective space, then $\mathrm{Pic}(M_{dP^4})= \mathbb{Z}[h].$ Therefore, $c_1 (\lambda_{CM})= r [h],$ where $r \in \mathbb{Z}$ and $h$ is the tautological class. This latter can be expressed as $h=c_1 (\tau), $ and $\tau$ is the tautological bundle of $S^m\mathbb{P}^1,$ i.e. $\tau = \mathcal{O}_{S^m \mathbb{P}^1}(1). $
Since there is a natural mapping from the base of our family, described in equation $\ref{fam}$, namely $f :\mathbb{P}^1 \rightarrow S^m \mathbb{P}^1$ that forgets the order of the eigenvalues of the pencil, we have that $f^{*} \mathcal{O}_{S^m \mathbb{P}^1}(1) = \mathcal{O}_{\mathbb{P}^1}(m).$  Hence,
\begin{align*}
\text{deg}(\mathcal{O}_{\mathbb{P}^1} (c)) &= \text{deg}(\lambda_{CM}) \cdot (\mathcal{X} \rightarrow \mathbb{P}^1) \\
&= \text{deg}(f^{*}\mathcal{O}_{S^m \mathbb{P}^1}(1)^r) \\
&= \text{deg}(\mathcal{O}_{\mathbb{P}^1} (mr)).
\end{align*}

Therefore, $r = c / m.$
%The volume form of $M_{dP^4}$ is given by

%$$vol(M_{dP^4}, \lambda_{CM}) = %c_1(\lambda_{CM})^{dim M_{dP^4}}.$$
%Thus,
%$$ vol(M_{dP^4}, \lambda_{CM})= \left( %\frac{c}{m} \right)^{m-3} (c_1 %(\mathcal{O}(1))^{m-3}. $$
In order to finish the computation of the volume we have to compute the top self intersection of $c_1(\mathcal{O}_{S^m \mathbb{P}^1}(1))$ descended to the moduli space.

Namely,

$$ Vol(M_{dP^4}, \lambda_{CM}) = \int_{M_{dP^4}} c_1(\mathcal{O}(1))^{m-3}. $$
Where $m-3= dim(M_{dP^4}).$
We can compute the above integral using the symplectic version of $M_{dP^4}$. We notice that the action of the maximal compact subgroup of $\mathrm{SL}(2)$, which is $\mathrm{SU}(2)$, is obtained from the irreducible representations of $SL(2)$ of degree $m+1.$ This action is Hamiltonian with moment map $\mu$ . By 
\cite[Corollary \ p.494]{GS} it follows that it is enough to prove that $\mu(x) = 0,$ thus the stabilizer $SU(2)_x$ is finite. Using the structure of the irreducible representation of $SL(2)$ over $S^{m} \mathbb{C}^2$ (which is in fact recalled below) we can conclude that $0$ is a regular value of the moment map and the symplectic reduction is a orbifold. In the GIT interpretation this means that the semistable locus is equal to the stable locus. Then, the Kirwan map
$ k: H^{\bullet}_{SU(2)}(S^m\mathbb{P}^1) \rightarrow H^{\bullet}(M_{dP^4})$ is surjective. Therefore, we can apply the nonabelian localization theorem \cite{JK95}. The latter, for the case of an Hamiltonian $\mathrm{SU}(2)$-action is as follows:

$$ k(\eta) e^{\omega_0}[M ] = \frac{n_0}{2} \mathrm{Res}_{X=0}\left( 4X^2 \sum_{F \in \mathcal{F}_+} \int_{F} \frac{i_F ^{*} \eta e^{\bar{\omega}}} {e_F (X)}dX \right), $$

 where $\mathcal{F}_+$ denotes the set of positive fixed cycles. Since the above action has only isolated fixed points, then $\mathcal{F}_+$ is given by a finite number of points. Thus, the integral on the right hand side reduces to a sum. If $\eta$ has degree equal to the real dimension of the quotient, then we can omit the exponential in the right hand side. 
In our case $\eta = c_1 ^{T} (\tau)^{m-3}$ , since $\tau$ is a \emph{prequantum} line bundle \cite[Definition 9.19]{Jef19} then  $i_{F}^{*}\eta$ is a multiple of the moment map evaluated at the fixed point $F$, see \cite[Lemma 9.31]{Jef19}. The action of the maximal torus $T=diag(t,t^{-1})$ on a generic element $p=p(u,v)$ of $S^m \mathbb{P}^1$ is given by 

$$ T \cdot p(u,v) = p(t^{-1}u, tv) = a_0 t^{-m}u^m + ... + a_k t^{2k-m}u^{m-k}v^k + ... + a_m v^m t^m. $$The symmetric polynomial space $S^m \mathbb{P}^1$ can be canonically identified with $\mathbb{P}^m$, therefore the above action can be specified as

$$ T \cdot (z_0, ..., z_m) = (t^{-m} z_0, ..., t^{2k-m} z_k, ..., t^{m} z_m) $$With this prescription, the fixed point set for this latter action is given by the standard basis $\{e_i \}_{i=0}^m$ of $\mathbb{C}^{m+1}$ .
The moment map $\mu_T : \mathbb{P}^m \rightarrow i\mathfrak{t}^{*}$ for this action is given by

$$ \mu_T (z_0 , ..., z_m) = \sum_{j=1}^m (m -2j) \frac{|z_j|^2}{|z|^2}x $$
Where $x$ denotes the basis of $i\mathfrak{t}^*$. The image of the fixed point $e_k$ is given by

$$ \mu_T (e_k) = (m-2k)x $$
Since all the fixed points are isolated and $\tau$ is a prequantum line bundle we conclude that, at any fixed point $e_k$ the numerator of the meromorphic function into the localization formula is given by

$$ i_{e_k}^{*} \eta= (-1)^{k} (m-2k)^{m-3} x^{m-3} $$
Since the maximal torus is identified with the one dimensional unitary group, the Euler class at the fixed point $e_k$ to the normal bundle is simply given by the products of the weight, that is

$$ e_{e_k}(X) = \prod_{j \neq k} 2(j-k) x^{m}= 2^m k! (m-k)! (-1)^{k} x^m$$
Thus, the meromorphic form is given by

$$ \frac{i_{e_k}^* \eta (X)}{e_{e_k}(X)} = \frac{(m-2k)^{m-3}}{2^m k! (m-k)! x^3} $$
  Therefore,  

\begin{align*} Vol(M_{dP^4}, \lambda_{CM}) &= \left( \frac{c}{m} \right)^{m-3} Res_{x=0} 4x^2 \left( \sum_{k:0<m-2k \leq m} \frac{(m-2k)^{m-3} x^{m-3}}{2^m k! (m-k)! x^m} \right)  \\
&= \left( \frac{c}{m} \right)^{m-3} Res_{x=0} \frac{1}{x} \left( \sum_{k: 0<m-2k \leq m} \frac{(m-2k)^{m-3}}{2^{m-2} k! (m-k)!} \right) \\
 \end{align*}

Hence, the volume of the moduli space of quartic del Pezzo's of odd dimension, with respect to the CM line bundle is given by

\begin{equation}  Vol(M_{dP^4}, \lambda_{CM}) = \left( \frac{c}{m} \right)^{m-3} \frac{1}{2^{m-1}}\left( \sum_{k: 0< m-2k <m} \frac{(m-2k)^{m-3}}{k! (m-k)!} \right).  \end{equation}
\end{dimostrazione} \hfill $\Box$

\begin{osservazione} \label{evencase} In the odd case, notice that the semistable locus does not equal the stable locus, therefore in order to apply the localization theorem we should use the Kirwan partial desingularization technique \cite{Kir85}. The semistable elements are fixed by nontrivial connected reductive subgroups of $\mathrm{SL}(2)$ and are those of the form $(z_1, ...,z_m)$ for which there exist distinct points $p$ and $q$ in $\mathbb{P}^1$ with exactly half of the points of the $m$-tuple equal to $p$ and the rest equal to $q$ . The partial desingularization it is achieved by blowing up, along the orbits corresponding to these latter points, the GIT quotient $S^m \mathbb{P}^1 // \mathrm{SL}(2).$ The result \ref{Voluquarti} should hold as well in the odd case. The reason of that should be in the nature of the $\mathrm{SL}(2)$ action on the symmetric power of $\mathbb{P}^1,$ namely the action is \emph{weakly balanced} (\cite{JK01} definition 15, examples 2.4 and 2.5). In \cite{JK01}(Theorem 25, Example 26) is shown that under the hypothesis of \emph{weakly balanced actions} then any intersection pairing on the partial desingularization can be localized in the undesingularized GIT quotient. \end{osservazione}

\section{The CM Volume of log Fano hyperplane arrangements moduli space}

In this section we will recall some generalites of the log Fano hyperplane arrangements and their moduli. Then, we will apply the Jeffrey-Kirwan non abelian localization theorem to compute the CM volume of the log Fano hyperplane arrangements moduli space

\subsection{Some generalities on the log Fano hyperplane arrangements and their moduli.}

In this section, we will give a brief description of the moduli space of log Fano hyperplane arrangements.

\begin{definizione}\cite{Fuji19b} \label{uno} Let $(X, D)$ be a $n-$ dimensional hyperplane arrangement, that is  $X= \mathbb{P}^n$ and $D$ is a divisor of $X$ of the form $\sum_{i=1}^m d_i H_i,$ where $d_i \in \mathbb{Q}_{>0},$ and the hyperplanes $H_i$ are mutually distinct $\forall i \in \{1,2, ..., m\}.$ 

\begin{enumerate}

\item  $(X,D)$ is said to be a \emph{log Fano hyperplane arrangement} if $(X, D)$ is a log Fano pair, i.e. $X$ is a projective variety over $\mathbb{C},$ $D$ is an effective $\mathbb{Q}-$divisor on $X$ such that, the pair $(X,D)$ is klt and $-(K_X + D)$ is ample.
\item $(X,D)$ is said to be a \emph{log Calabi-Yau hyperplane arrangement}, if $X$ is a projective variety over $\mathbb{C},$ $D$ is an effective $\mathbb{Q}-$divisor on $X$ such that, the pair $(X,D)$ is lc and $K_X + D$ is $\mathbb{Q}-$linearly equivalent to $0.$ 
\item $(X,D)$ is said to be a \emph{log general type hyperplane arrangement}, i.e. $X$ is a projective variety over $\mathbb{C},$ $D$ is an effective $\mathbb{Q}-$divisor on $X$ such that, the pair $(X,D)$ is lc and  $K_X + D$ is ample.
\end{enumerate}
\end{definizione}

Consider a $(X,D)$ $n$-dimensional hyperplane arrangement like in \ref{uno}. From the above conditions, the choices of the weights $d_i \in \mathbb{Q}_{>0}$ determine the \emph{nature} of the hyperplane arrangement, namely one of the kind defined in \ref{uno}, as the following easy result shows

\begin{proposizione} Let $(X,D)$ be a $n$-dimensional hyperplane arrangement like in \ref{uno}. Then

\begin{itemize}

\item $(X,D)$ is a klt log Fano hyperplane arrangement if and only if $$\sum_{i=1}^m d_i < n+1.$$ \item $(X,D)$ is a lc log Calabi-Yau hyperplane arrangement if and only if  $$\sum_{i=1}^m d_i = n+1.$$ \item $(X,D)$ is a lc log general type hyperplane arrangement if and only if $$ \sum_{i=1}^m d_i > n+1. $$\end{itemize}
\end{proposizione}

Let $(X, D)$ be a klt log Fano hyperplane arrangement. A hyperplane $H_i \subset \mathbb{P}^n$ correspond to a point $p_i$ in the dual projective space $\mathbb{P}^{n*}$, or equivalently in the Grassmannian  $\mathrm{Gr}(n-1,n).$  Let $p:=(p_1 , ..., p_m) \in (\mathbb{P}^{n*})^m,$ the group $ \mathrm{SL}(n+1)$ acts naturally on $(\mathbb{P}^{n*})^m$ via its linear representation in $\mathbb{C}^{n+1}$ \cite{Dol03}. Some positive multiple of the line bundle $\mathcal{L}_d = \mathcal{O}(d_1, ..., d_m)$ admits a unique $\mathrm{SL}(n+1)-$linearization. The moduli space of log Fano hyperplane arrangement can be described in a GIT fashion. The notion of stability on the point $p \in (\mathbb{P}^{n*})^m$ with respect to $\mathcal{L}_d$ is the \emph{classical} one, given in the standard literature \cite{MFK94}, \cite{Dol03}. It is proven by Fujita in \cite[Theorem 1.5, Corollary 8.3]{Fuji19b} that the above GIT definitions matches with the K-stability conditions. Therefore, the K-moduli of log Fano hyperplane arrangements can be identified with the GIT quotient

$$ M_d \simeq (\mathbb{P}^{n})^m //_{\mathcal{L}_d} \mathrm{SL}(n+1). $$

In the above GIT setting, with some abuse of notation, we mean a positive multiple of $\mathcal{L}_d$ and thus we are considering $d_i \in \mathbb{Z}, \forall i \in \{1,2,...,m\}$ such that $\sum_{i=1}^m d_i < n+1.$

\subsection{The intersection number of the log CM line bundle} 

In this section we compute the intersection number of the log CM line bundle with a family of log Fano hyperplane arrangements with base $\mathbb{P}^1.$ Fix $m$ generic hyperplane in $\mathbb{P}^n, H_i$ for $i = 1,2, \ldots, m$ . Consider the product $\mathcal{X} := \mathbb{P}^n \times \mathbb{P}^1,$ We define a divisor $ \mathcal{D} \subset \mathcal{X},$ as 
$$\mathcal{D} = \sum_{i=1}^{m-1}d_i H_i + d_m \lambda H_m$$

Where $c_1(H_i) = pr_1^{*}h,$ for $i = 1,2, \ldots, m-1$ where $h$ is the hyperplane class of $\mathbb{P}^n$ and similarly $c_1(\lambda H_m) = pr_2^{*} l$ where l is the hyperplane of $\mathbb{P}^1$ and $\lambda \in P^1$ . The $\mathrm{pr}_i, \ i=1,2$ denotes the projections onto the first and second factor of $\mathcal{X}.$ We can think about $\mathcal{D}$ as $m-1$ fixed hyperplanes of $\mathbb{P}^n$ together with a line $l$ with weight $d_m$  free to move along the \emph{diagonal} of $\mathcal{X}.$ We assume that $d_i \in (0,1) \cap \mathbb{Q}, \ \forall i \in \{1,2, .., m\},$ $\sum_{i=1}^{m} d_i < n+1,$ and the chosen line $l$ and hyperplanes are in general position. 

\begin{proposizione} \label{prop1} With the above data we have

$$c_1(\lambda_{CM, \mathcal{D}}) = (n+1)d_j \left(n+1 - \sum_{i=1}^m d_i \right)^n, \ \forall j \in \{1,2, ...,m\}. $$
\end{proposizione}

\begin{dimostrazione} We see that the projection $\mathrm{pr}_2 : \mathcal{X} \rightarrow \mathbb{P}^1$ gives a proper and flat family of relative dimension $n.$ Choose $\mathcal{L}= -K_{\mathcal{X}/B} - \mathcal{D}$ and use assertion $1.$ of \ref{cpsgm}, we find that
$$ c_1(\mathcal{L})= \left( n+1 - \sum_{i=1}^m d_i \right) \mathrm{pr}_1^{*}h - d_m \mathrm{pr}_2^{*}l $$
Using the binomial expansion, the only term that survives is 

$$c_1(\lambda_{CM,\mathcal{D}}) = -\mathrm{pr}_{2*} c_1 (\mathcal{L})^{n+1} = (n+1)d_m  \left( n+1 - \sum_{i=1}^m d_i \right)^n \mathrm{pr}_{2*} \left( \mathrm{pr}_1^{*} h \cdot \mathrm{pr}_2^{*}l \right). $$ By using the projection formula and the arbitrariness of the choice of the weight on the line, the claim follows. \hfill $\Box$

\end{dimostrazione}

We know that $\mathrm{Pic} (M_d) = \mathbb{Z}^m,$ see for example \cite[Chapter 11, Lemma 11.1]{Dol03}.  Therefore $c_1 (\lambda_{CM, \mathcal{D}}) \in \mathrm{Pic}(M_d) \Rightarrow c_1(\lambda_{CM, \mathcal{D}})= \mathcal{O}(r_1, ..., r_m),$ for some $(r_1, ..., r_m) \in \mathbb{Z}^m.$ In order to compute the $r_j'$s we set

$$ c_1 (\lambda_{CM, \mathcal{D}}) \cdot (\mathcal{X} \rightarrow \mathbb{P}^1) = k, $$
that means, $\forall j \in \{1,2, ..., m\}$ 

$$ \int_{(\mathbb{P}^1)_j} \mathcal{O}(r_1, ..., r_m) = r_j=k .$$
Therefore, from \ref{prop1} we get $r_j = (n+1)d_j \left( n+1 - \sum_{i=1}^m d_i \right)^n.$ This proves the following

\begin{proposizione} \label{prop3} In the above setting, the log CM line bundle on the moduli space of log Fano hyperplane arrangement is given by

$$c_1 (\lambda_{CM, \mathcal{D}})=\left( n+1 - \sum_{i=1}^m d_i \right)^n \mathcal{O}((n+1)d_1, ..., (n+1) d_m)$$

\end{proposizione}

\subsection{The volume of the moduli space of weighted hyperplane arrangements} 

Consider the GIT quotient

\begin{equation} \label{gittot} M_d = (\mathbb{P}^n)^{m} //_{\mathcal{L}_d} \mathrm{SL}(n+1)(\mathbb{C}) \end{equation}
with linearization $\mathcal{L}_d= \mathcal{O}(d_1 , ..., d_m).$ A point in $M_d$ corresponds to a log Fano hyperplane arrangement. $(\mathbb{P}^n, \sum_{i=1}^m d_i H_i).$ When a compact group $K$ acts on a symplectic manifold $X$ and $\mu \colon X \rightarrow \mathfrak{k}^{*}$ is a moment map for the action, the symplectic form on $M$ induces a symplectic form on the quotient $\mu^{-1}(0)/K$, away from its singularities, which is the Marsden-Weinstein reduction or symplectic quotient of $M$ by the action of $K.$
In this setting, the symplectic quotient can be identified with a geometric invariant quotient $X//G,$ where $G$ is such that $K^{\mathbb{C}}=G,$ where $K^{\mathbb{C}}$ is the complexification of K. The details of this correspondence can be found in \cite[Remark 8.14]{Kir84}. This correspondence applies immediately to Equation \ref{gittot}, indeed $(\mathbb{P}^n)^m$ is a symplectic manifold, and  $\mathrm{SU}(n+1)^{\mathbb{C}}=G$. The moment map is well known in the literature, see for example \cite[Example 5.5]{SzeBook}. 

\begin{teorema}  (Residue Theorem) \cite[Theorem 8.1]{JK95}, Let $(M,\omega)$ be a  symplectic manifold with a  Hamiltonian action of a compact Lie group $G$ with moment map $\mu$. Assume that the center of $\mathfrak{g}$ is a regular value of the moment map. Denote by $M_{red}$ the corresponding symplectic reduction. Denote by $K$ the Kirwan map. Then, given a formal equivariant class $\eta e^{\bar{\omega}} \in H^{\bullet}_G (M)$ 

\begin{equation} \label{jkr} k( \eta e^{\bar{\omega}})[M_{red}]= n_0 C^G \mathrm{Res} \left(\mathcal{D}^2(X) \sum_{F \in \mathcal{F}} H_F^{\eta}(X) [dX] \right) \end{equation}

where $n_0$ denotes the order of the stabilizer in $G$  of a generic point in  $\mu^{-1}(0)$, the constant $C^G$ depends on the volume of the maximal torus, on the Weyl factor and on the positive roots, namely

$$ C^G = \frac{(-1)^{s+n_+}}{|W| vol(T)} $$
Where $s= \mathrm{dim}G$, $n_+ = \frac{s-l}{2}$ is the number of positive roots, $l= \mathrm{dim}T$ and $W$ is the Weyl group of $G$. The factor $\mathcal{D}(X)$ is called the \emph{Weyl factor} and is defined as the product of the positive roots of $G$. The $H_F ^{\eta} (X)$ are given by

$$ H_F^{\eta}(X) = e^{\mu(F)} \int_F \frac{i_F^{*}\eta(X) e^{\omega}}{e_F (X)}. $$
Here F is a connected component of the set of fixed points for the action of the maximal torus of $G$ and $\mathcal{F}$ is the set of such connected components. 
\end{teorema}

Theorem 5.3.6 of \cite{Ale15} gives the necessary conditions for applying the Jeffrey-Kirwan residue theorem to $M_d.$  For computing the volume, it is sufficient to set $\eta=1$ in \ref{jkr}, so that the terms $H_F^{\eta}$ becomes simply

$$H_F(X):= H_F^{1}(X) = \frac{e^{\mu(F)(X)}}{e_F(X)} $$
Indeed, since the fixed points for this action are isolated, then the connected components of $\mathcal{F}$  for the maximal torus action consist of fixed points $F$, hence $i^{*}_F \eta(X)e^{\omega}= (e^{\omega})_0 (F) =1.$ Furthermore, we chose a normalization for which $Vol(T)=1.$ Hence the formula to apply for calculating the volume will be the following

 \begin{equation} \label{final} k( 1.e^{\bar{\omega}})[M_{red}]=  \frac{n_0 (-1)^{s+n_+}}{|W|} \mathrm{Res} \left(\mathcal{D}^2(X) \sum_{F \in \mathcal{F}} \frac{e^{\mu(F)(X)}}{e_F (X)} [dX] \right). \end{equation}
The fixed points for the maximal torus action on each $\mathbb{P}^n$ are the standard basis elements of $\mathbb{C}^{n+1}$, which will be denoted as usual $\{e_i\}_{i=1}^{n+1}$. Therefore we have $(n+1)^{m}$ fixed points in $(\mathbb{P}^n)^m$, each of them is a string of elements of the standard basis. For the sake of clarity, we define a general notation that will be used henceforth. 

\begin{itemize}
\item $d=(d_1, ..., d_m) \in \mathbb{R}_+^m$
\item $\delta := \sum_i d_i $ 
\item $[m]=\{1,2,...,m\}$
\item If $I \subset [m],$ $D_{I}:= \sum_{i \in I} d_i$
\item $n:= \mathrm{dim}\mathbb{P}^n$
\item $ f:=(f_1, ..., f_m) \in \mathcal{F}:=[n+1]^m$,  is identified with a fixed point in $(\mathbb{P}^n)^{m}$for the $T$-action.
\item If $j \in [n+1],$ call $I_j(f) := \{ i \in [n] : f_i =j \},$ $m_j(f):=|I_j(f)|$ 
\item $\sum_{i=1}^{n+1} m_j(f) = m$
\item $ \delta_j (f) := D_{I_j (f)} = \sum_{i \in I_j (f)} d_i$,  sometimes this latter will be just denoted by $\delta_j.$ 
\end{itemize}

Set $h_f(X) := \mathcal{D}^2(X) \frac{e^{\mu(F)(X)}}{e_F (X)}$.  

\begin{proposizione}\label{generale} The function $h_f(X)$ with respect to the open cone $\Lambda_+$ (\cite[Proposition 3.2]{JK2}) associated to the maximal torus of $SU(n+1)$ is given by

$$ h_f(X) = (-1)^{m\cdot (n+1) -\sum_{j=1}^{n+1} m_j(f) \cdot j }\frac{e^{\sum_{i=1}^{n+1} \left( \sum_{k=1}^i \left( 1 - \frac{i}{n+1}\right) \delta_k (f) - \sum_{k=i+1}^{n+1} i \frac{\delta_k (f)}{n+1}\right) \alpha_i.}}{\prod_{ 1 \leq i < j \leq n+1} \alpha_{ij}^{m_j + m_i -2} }. $$
Where, $\alpha_{ij} = \epsilon_i - \epsilon_j,$ $\alpha_i = \alpha_{i,i+1},$ and the $\epsilon_i$ are elements of the basis of the dual of the Lie algebra of the maximal torus of $SU(n+1).$
\end{proposizione}

\begin{dimostrazione} The \emph{Weyl factor} is $ \mathcal{D}^2 (X) = \prod_{i <j} (\epsilon_i - \epsilon_j)^2.$ The $T$-moment map on each $\mathbb{P}^n$ is simply given by 

\begin{align*} \mu_j (z_1, ..., z_{n+1}) &= \sum_{i=1}^{n+1} \frac{|z_i|^2}{|z|^2} \epsilon_i  \\ &= \sum_{i=1}^n \left(  \sum_{k=1}^i \frac{|z_i|^2}{|z|^2} - \frac{k}{n+1} \right) \alpha_i. \end{align*}

The \emph{total} moment map of the product manifold $\mathbb{P}^n$ is given by the weighted sum of the moment maps over each factor of $(\mathbb{P}^n)^m$, namely

$$ \mu(z^1 , ..., z^j) = \sum_{j=1}^m d_j \mu_j (z^j).$$
We recall that, via the Killing form, a basis for the dual of the Lie algebra of the maximal torus of $SU(n+1)$ can be chosen such that it coincides with the dual basis $\{\epsilon_i\}_{i=1}^{n+1}$ of the Euclidean space $\mathbb{R}^{n+1*},$ with the condition that the sum of all the elements of such a basis is equal to zero. Using this basis, we can express the positive roots $\Delta_+$ of $SU(n+1),$ as follows

$$ \Delta_+ := \{ \epsilon_i - \epsilon_j \ | \ 1 \leq i < j \leq n+1 \}.$$ Define $\alpha_{ij}= \epsilon_i - \epsilon_j,$ and $\alpha_i = \alpha_{i,i+1}.$

On the fixed point $f,$ we have

\begin{equation} \label{tino}\mu(f) = \sum_{j=1}^m d_j (f) \mu_j(e_j)= \sum_{s=1}^{n+1} \delta_s (f) \mu(e_s). \end{equation}
Moreover,

\begin{displaymath}
 \mu(e_s) = \sum_{i=1}^k \frac{|z_i|^2}{|z|^2} - \frac{k}{n+1}=
\begin{cases}
 -\frac{k}{n+1}& \ \text{if $s>j$} \\
 1- \frac{k}{n+1}  & \ \text{if $s\leq j.$}
 \end{cases}
\end{displaymath}
Then,
\begin{gather*}  \mu(e_s) = \frac{1}{n+1} \bigl ( - \alpha_1 - ... - (s-1) \alpha_{s-1}  +
\\  \nonumber  + (n+1 -s) \alpha_s + (n-s) \alpha_{s+1} + ... + 1 \cdot \alpha_n \bigr).
 \end{gather*}
The coefficient $\alpha_i$ is $\frac{-i}{n+1},$ when $i <s$, and when $i \geq s$ is $\frac{n+1-i}{n}.$ Therefore the coefficient of $\alpha_i$ in $\mu(f)$ is given by 

\begin{equation} \label{mino} \sum_{k=1}^i \left( 1 - \frac{i}{n+1}\right) \delta_k (f) - \sum_{k=i+1}^{n+1} i \frac{\delta_k(f)}{n+1}. \end{equation}
Therefore, we get

$$\mu(f)=\sum_{i=1}^{n+1} \left( \sum_{k=1}^i \left( 1 - \frac{i}{n+1}\right) \delta_k (f) - \sum_{k=i+1}^{n+1} i \frac{\delta_k (f)}{n+1}\right) \alpha_i. $$
Comparing with \ref{tino}, we get 

%\begin{align*} 
%\mu(f) &= \sum_{s=1}^{n+1} \delta_s \left( \sum_{i=s}^n \left( 1 - \frac{i}{n+1}\right) \alpha_i  + \sum_{i \leq s-1}  \frac{-i}{n+1} \alpha_i \right) \\
 %&= \sum_{s} \sum_{i \leq s-1} \delta_s \left(\frac{-i}{n+1} \right) \alpha_i + %\sum_{s} \sum_{i \geq s} \delta_s \left( 1 - \frac{i}{n+1} \right) \alpha_i \\
 %&= \sum_{i=1}^n \left[ \sum_{i \leq s-1}\left(  \frac{-i}{n+1} \right)  \delta_s \right]\alpha_i +\sum_{i=1}^{n} \left[ \sum_{i \geq s} \left( 1 - \frac{i}{n+1} \right) \delta_s \right] \alpha_i \\
 %&= \sum_{i=1}^n \left[ - \sum_{s>i} \frac{i \delta_s}{n+1} + \sum_{s \leq i} \left(1 - \frac{i}{n+1} \right) \delta_s \right] \alpha_i.
%\end{align*}
%Hence,

$$\mu(f) = \sum_{i=1}^n \left[ \left(1 - \frac{i}{n+1}\right) \cdot\sum_{k=1}^i \delta_k - \frac{i}{n+1}\cdot \left( \sum_{k=i+1}^{n+1} \delta_k \right) \right] \alpha_i. $$ 

Now we calculate the Euler class of the normal bundle at the generic fixed point $f.$ The equivariant Euler class of the normal bundle to $f$ is defined as:

$$ e_f(X) = \prod_j \left(c_1 (\nu_{f,j}) + \beta_{f,j}(X) \right) $$Where the $\nu_{f,j}$ are the line bundles of the formal splitting of the normal bundle to $f$ given by the \emph{splitting principle} and the $\beta_{f,j}$ are the weights of the torus action at the fixed point $f.$ 
Observe that in our case the terms $c_1 (\nu_{f,j})$ contributes to zero,  and since the connected components of the fixed points of the action are points, then the normal bundles at these fixed points are simply given by tangent spaces. Thus,

$$\label{subdivision} \nu_f = \bigoplus_{m_1} T_{e_1} \mathbb{P}^n \oplus \bigoplus_{m_2} T_{e_2} \mathbb{P}^n \oplus \ ... \ \oplus \bigoplus_{m_{n+1}} T_{e_{n+1}}\mathbb{P}^n $$
In order to calculate the weights on the normal bundle we can take the chart on $\mathbb{P}^n$ such that $z_1 \neq 0$, then we define new coordinates $u_1= \frac{z_2}{z_1}$ and $u_2 = \frac{z_3}{z_1}, ..., \frac{z_{n+1}}{z_1}$. Thus $T \cdot (u_1 , u_2, ..., u_n) = \left( e^{t_2 - t_1}u_1, e^{t_3 -t_1}u_2, ..., e^{t_{n+1} - t_1}u_n \right)$. Identifying the $t_i$ with the $\epsilon_i$ and noticing that 

$$ \sum_{j=k}^{i-1} \alpha_j = \epsilon_k - \epsilon_i, \ \mathrm{for} \ k= 1, ..., i-1, $$
$$ -\sum_{j=i}^{k-1} \alpha_j = \epsilon_k - \epsilon_i, \ \mathrm{for} \ k= i+1, ..., n+1, $$

we find that the equivariant Euler class at $e_1$ is given by 

$$ e(T_{e_1}\mathbb{P}^n) = (-1)^{n} \prod_{k=1}^{n+1} \left( \sum_{j=1}^{k-1} \alpha_j\right). $$
Similarly,  we find, for all $s \in \{1,2, ..., n+1 \}$

\begin{equation}\label{pino} e(T_{e_s}\mathbb{P}^n)=(-1)^{n-s+1} \prod_{k=1}^{s-1} \left( \sum_{j=k}^{s-1} \alpha_j \right) \cdot \prod_{k=s+1}^{n+1} \left( \sum_{j=s}^{k-1} \alpha_j \right). \end{equation}
Since

$$ \sum_{j=k}^{s-1} \alpha_j = \alpha_{ks}, \ \mathrm{if} \ k \leq s-1, $$
$$ \sum_{j=s}^{k-1} \alpha_j = \alpha_{sk}, \ \mathrm{if} \ k \geq s+1, $$
\ref{pino} becomes 

\begin{equation} \label{gino} e(T_{e_s}\mathbb{P}^n)=(-1)^{n-s+1} \prod_{k=1}^{s-1} \alpha_{ks} \cdot \prod_{k=s+1}^{n+1} \alpha_{sk}. \end{equation}
Let $f$ be a fixed point, then the Euler class at that point is 

\begin{equation} \label{rino} e_f (X)=\prod_{s=1}^{n+1} e(T_{e_s}\mathbb{P}^n)^{m_s (f)}. \end{equation}
Using \ref{gino} we can compute \ref{rino}. Firstly, notice that, since $\sum_{s=1}^{n+1} m_s (f) = m$ we have the following

\begin{align*} 	
\sum_{s=1}^{n+1} m_s (f) (n-s+1) &= (n+1) \sum_{s=1}^{n+1} m_s (f) - \sum_{i=1}^{n+1}m_s (f) \cdot s \\ 
&= (n+1)m - \sum_{s=1}^{n+1} m_s(f) \cdot s.
\end{align*}
For later use, we define

$$ \epsilon= (-1)^{(n+1)m - \sum_{j=1}^{n+1} m_j(f) \cdot j}. $$
It remains to calculate the exponent of the remaining factors in \ref{gino}. In order to do so, suppose we take the positive root $\alpha_{ij} := \epsilon_i - \epsilon_j,$ with $i < j.$ It appears in \ref{gino} exactly one time if $s=j$, and exactly one time if $i=s.$ Hence, the exponent of $\alpha_{ij}$ is given by $m_j + m_i.$ Thus, 
 
$$ \label{EulerNormal} e_f (\nu_f) =  \epsilon \cdot \prod_{1 \leq  i < j \leq n+1} \alpha_{ij}^{m_j + m_i}. $$
Finally, noticing that the Weyl factor is the product of positive roots, by using the definition of $h_f(X)$ the claim follows. \hfill $\Box$

\end{dimostrazione}

In the next two sections we calculate explicitly the GIT volume for $n=1,$ and $n=2.$ 

\subsection{The dimension one case}

The volume formula for $M_d$, in the one dimensional case, is calculated in many works \cite{Man08}, \cite{Ta01}, \cite{Khoi05}, the case of $d=1$ was calculated in \cite{Mar00}. We will now provide a simple proof of the $M_d'$s volume which is coherent with the cited works. In the dimension one case $M_d$  can be described as the following GIT quotient:
$$ M_d \simeq (\mathbb{P}^{1})^m //_{\mathcal{L}_d} \mathrm{SL}(2). $$
Where $\mathcal{L}_d$ is the linearization mentioned in Section 4.1.
\begin{teorema} \label{one} The volume of $M_d$ with $d_i \in (0,1) \cap \mathbb{Q}$, for all $i \in \{1,2,....,m\},$ and $m \geq 4$ is given by

\begin{equation}  \mathrm{vol}(M_d) = -\frac{1}{2(m-3)!}  \sum_{f \in \mathcal{F}_+} (-1)^{m_1(f)} \left(\delta_{1}(f) - \delta_{2}(f)\right)^{m-3}  \end{equation}
Where

$$\mathcal{F}_+ = \{ f \in \mathcal{F}=[2]^m |\delta_{1}(f) - \delta_2(f)  > 0 \}  $$
\end{teorema}

\begin{dimostrazione} By a direct application of \ref{generale} we find that  the general term of the residue $h(X)$ is given by   
\begin{equation} \label{step1}
 h(X) =\frac{e^{\mu(F)(X)}}{e_F (X)} = (-1)^{m_1 (f)} \frac{e^{\lambda \alpha_1}}{\alpha_1^{m-2}}.\end{equation}
Where in the Euler class we used the identity $m_1(f) + m_2(f) = m.$ 
We see that the given meromorphic function $\frac{e^{\lambda} \alpha_1}{\alpha_1 ^{m-2}}$ has a pole in zero with mutiplicity $m-2.$ Then the residue at zero is given by

  \begin{equation} \label{unires}
 \mathrm{res} _{\alpha_1=0}^+ \left( \frac{e^{\lambda \alpha_1}}{\alpha_1^{m-2}}\right) :=
 \begin{cases}
\frac{\lambda^{m-3}}{(m-3)!} & \ \text{if $\lambda >0$} \\
 0 & \ \text{else.}
 \end{cases}
 \end{equation}

The coefficient $\lambda,$ namely the moment map at $F,$ can be easily computed along the line of the computation carried at the beginning of the proof of \ref{generale} to find

\begin{equation} \label{lambdaunires} \mu(f) = \frac{1}{2} (\delta_1 (f) - \delta_2 (f)) \alpha_1, \end{equation}
therefore  $\lambda = \frac{1}{2} (\delta_1 (f) - \delta_2 (f)).$ The Weyl group $W$ has dimension $2$ and therefore the constant $C^G$ is equal to $-1/2.$  Since $n_0 = 2^{m-3},$ by taking the sum over $\mathcal{F}_+$ and combining equation \ref{step1} and \ref{unires} the claim follows immediately.

\hfill $\Box$

\end{dimostrazione}

\subsection{The dimension two case}

Consider the GIT quotient 

$$ M_d = (\mathbb{P}^2)^m //_{\mathcal{L}_d}  \mathrm{SU}(3) $$
with linearizzation $\mathcal{L}= \mathcal{O}(d_1 , ..., d_m)$. Where $d_i \in (0,1) \cap \mathbb{Q}, \forall i \in \{1,2 ..., m\}$.   Points in $M_d$ are pairs $(X, D)$ where $X$ is a copy of $\mathbb{P}^2$ and $D$ is a divisor in $\mathbb{P}^2$ identified as the weighted sum of the hyperplane classes $H_i,$ that is $D = \sum_{i=1}^m d_i H_i$. Furthermore, we require for each pair $(X, D)$ that the relative log anticanonical bundle $L= -(K_X + D)$ is ample. This latter leads to require that $\sum_{i=1}^m d_i < 3$ and therefore every pair is a log Fano pair.

%%%%%%%%%%%%%%%%%%%%%%%%%%%%%%%%%% da qua in poi

Fix the following notation.
Set $\xi_i : = \delta_i - \delta /3$ for $i=1,2,3$. Then
$\xi_1 + \xi _2 + \xi _3 = 0$ and
\begin{gather}
  \label{lambdazzi}
  \begin{aligned}
    & \lambda_1 = \frac{2 \delta_1 - \delta_2 - \delta_3}{3}= \delta_1 -
    \delta /3 = \xi_1 , \\
    &\lambda_2 =\frac{\delta_1 + \delta_2 - 2
      \delta_3}{3} =\delta/3 - \delta_3 = - \xi_3= \xi_1 + \xi_2, \\
    & \lambda_1 - \lambda_2 = \delta/3 - \delta_2 = -\xi_2.
  \end{aligned}
  \\
\begin{gathered}
  \mathcal{A}(d) = \{ f \in [3]^m \ | \lambda_1 > \lambda_2 > 0\} = \{
  f \in [3]^m: \xi_2 <0, \xi_3<0\},
  \\
  \mathcal{B}(d) = \{ f \in [3]^m | \ \lambda_2 > \lambda_1 > 0 \} =
  \{ f \in [3]^m | \ \xi_{1}>0, \ \xi_{2} >0 \} .
\end{gathered}
\end{gather}

This section is dedicated to the proof of the following result.

\begin{teorema} \label{due} The GIT volume of the moduli space $M_d$ of
  line arrangements in $\mathbb{P}^2$ is given by
  \begin{gather*}
    \mathrm{Vol}(M_d)  = \\
    =
    % \frac{-1}{6(2m-8)!}
    -\sum_{f \in \mathcal{A}(d)} \frac{ (-1)^{m_2(f)} }{6(2m-8)!}
    % \left(
    \sum_{j=0}^{2m-8} \binom{2m-8}{j}\binom{m+m_2-6-j}{m_2 + m_3 -3} \xi_2^j \xi_3^{2m-8 -j} +\\
    % \left(\frac{\delta}{3} - \delta_3 (f) \right)^{j} \left(
    %   \delta_1(f) - \frac{\delta}{3} \right)^{2m-8-j} \right) \\
    - % \frac{-1}{6(2m-8)!}
    \sum_{f \in \mathcal{B}(d)} \frac{ (-1)^{m_2(f)} }{6(2m-8)!}
    % \left(
    \sum_{j=0}^{2m-8} \binom{2m-8}{j}\binom{m+m_2 -6-j}{m_1 + m_2 - 3}
    \xi_1^j \xi_2^{2m-8-j}.
    % \left( \frac{\delta}{3} - \delta_3(f) \right)^{j}
    % \left(\frac{\delta}{3} - \delta_2 (f) \right)^{2m-8-j} \right)
  \end{gather*}
\end{teorema}
\begin{dimostrazione} 
We start recalling some elementary combinatorial identities.
The \emph{falling factorial powers} are defined as follows
\cite[p.47-48]{gkp}: for $k\in \mathbb{Z}, k\geq 0$ and $z \in \mathbb{C}$ set
\begin{gather}
  z\falfa{k}:=
  \begin{cases}
    1 & \text{if } k=0\\
    \prod_{j=0}^{k-1} (z - j) & \text{otherwise}.
  \end{cases}
\end{gather}
We set $0^0 : = 1$, so the function $x^0$ is identically $=1$ on the
real line.  Thus for any $\xi \in \mathbb{R} $ and $k \in \mathbb{Z}$,
$k\geq 0$ we have
\begin{gather}
  \label{derivate}
  D^k(x^\xi) = \xi\falfa{k} x^{\xi - k} .
\end{gather}
For $z\in \mathbb{C}$ and $k\in \mathbb{Z}$ set \cite[p. 154]{gkp}:
\begin{gather}\label{fattore}
  \binom{z}{k} : = \begin{cases} \frac{z\falfa{k}}{k!} & \text{if } k\geq 0 ;\\
    0 & \text{if } k<0.
  \end{cases}
\end{gather}
We have
\begin{gather}
  \notag z\falfa{k} = (-1)^k (k-z-1)\falfa{k},\quad
  \label{nupper}
  \binom{z}{k} = (-1)^k \binom {k-z-1}{k} .
\end{gather}
(See \cite[p.164]{gkp}.)
\begin{gather}
  \label{bibin}
  \sum_{j} (-1)^j \binom{s +j }{n}\binom{l}{m + j} = (-1)^{l+m}
  \binom{s-m}{n-l}.
\end{gather}
See formula (5.24) in \cite[p. 169]{gkp}.  If $l \geq 0$, $s, m $ are
integers, then
\begin{gather}
  \label{tribin}
  \sum_{j} \binom{l}{m + j} \binom{s }{n+j} = \binom{l+s}{l-m+n}.
\end{gather}
See (5.23) in \cite[p. 169]{gkp}.  If $n > m \geq 0$, then
\begin{gather}
  \label{zero}
  \binom{m}{n}= 0
\end{gather}
For integers $n, m\geq 0$ we have
\begin{gather}
  \label{gkpp}
  (-1)^m\binom { -n -1 }{m} = (-1)^n \binom { -m -1}{n}
\end{gather}
See (5.15) of \cite[p. 165]{gkp}.
\begin{gather}
  \label{gkp2}
  \binom{-n}{k} = (-1)^k \binom {n+k-1}{k}
\end{gather}
This is (5.14) of \cite[p. 164]{gkp}.

Since we are specializing to $n=2$ we have $f \in [3]^m$ the formula in Proposition \ref{generale}
becomes
\begin{gather}
  \label{aa1} h_f = (-1)^ { m - \sum_{j=1}^3 j \, \cdot\, m_j(f) }
  \frac{ e^{\lambda_1 \alpha_1 + \lambda_2 \alpha_2} }
  { \alpha_1 ^{ m_1 +m_2 -2 } \alpha_2 ^{m_2 + m_3 -2 } (\alpha_1 + \alpha _2 ) ^{ m_1+ m_3 -2}}\\
  \notag \lambda_1 = \lambda_1(f,d) : = \frac{2 \delta_1 - \delta_2 - \delta_3}{3}  \\
  \notag \lambda_2 = \lambda_2(f,d) : = \frac{\delta_1 + \delta_2 - 2
    \delta_3}{3}.
  % \notag \epsilon(f):= (-1)^ { m - \sum_{j=1}^3 j \, \cdot\, m_j(f) }.
\end{gather}
We use $\{x=\alpha_1,y=\alpha_2\}$ as a basis.  For
$\lambda_1 , \lambda_2 \in \mathbb{R}$ and $a,b,c \in \mathbb{Z}$, set
\begin{gather}
  \notag R(\lambda_1, \lambda_2, a,b,c,):= \res^\lambda \biggl[\frac{e^{\lambda_1 x+
      \lambda_2 y
    }}{ x^{a} y^{b} (x+y)^{c}}\biggr] = \\
  =
  \label{ameno1} R(\lambda_1, \lambda_2, a,b,c,) = \res^+_x \res^+_y \biggl[
  \frac{e^{\lambda_1 x+ \lambda_2 y }}{ x^{a} y^{b} (x+y)^{c}} \biggr] =
  \res^+_x \frac{ e^ { \lambda_1 x} }{x^a} \res^+_y \biggl[ \frac{e^{
      \lambda_2 y }}{ y^{b} (x+y)^{c}} \biggr].
\end{gather}
The following is well-known.
\begin{lemma}
  \label{minilemma}
  If $g$ is a meromorphic function which is holomorphic at $z_0$, then
  \begin{gather}
    \label{resdev}
    \res_{z=z_0} \frac{g(z)}{(z-z_0)^k} = \frac{ D^{k-1}g(z_0)
    }{(k-1)!} \quad \text{if } k\geq 1
  \end{gather}
  and vanishes otherwise.
\end{lemma}
For $n\in \mathbb{Z}$ set
\begin{gather}
  \phi_n(x):= \frac{ e^x}{x^n } =x^{-n} e^x .
\end{gather}
By Leibniz formula for higher derivatives we have
\begin{gather}
  \label{Dkphin}
  D^k \phi_n = D^k (x^{-n} e^x ) = e^x \sum_{j=0}^k \binom{k}{j}
  (-n)\falfa{j} x^{-n-j} .
\end{gather}
For $\lambda, a \in \mathbb{R}$ and $n\in \mathbb{Z}$ set
\begin{gather*}
  g_{\lambda,a,n}(z) := \frac{e^{\lambda z} }{(z-a)^n} = e^{\lambda a} \, \lambda^n
  \phi_n(\lambda z - \lambda a).
\end{gather*}
Observe that if $f$ is a function, $\lambda, b $ are constants and
$h(x):= f(\lambda x + b)$, then for any $k$
\begin{gather*}
  D^k h (x) = \lambda^k \cdot (D^k f) (\lambda x + b).
\end{gather*}
Thus
\begin{gather}
  \label{Dg}
  \begin{gathered}
    \bigl(D^k g_{\lambda,a,n}\bigr)( z) = \lambda^{n+k} \, e^{\lambda a}\,
    \bigl(
    D^k\phi_{n}\bigr)(\lambda z -\lambda a) =\\
    = \bigl(D^k g_{\lambda,a,n}\bigr)( z) = e^{\lambda z} \sum_{j=0}^k \binom
    {k}{j} (-n)\falfa{j} \lambda^{k-j} (z-a)^{-n -j }.
  \end{gathered}
\end{gather}
Let us adopt the following convention: if $P$ is an inequality, then
$[P]$ is equal to 1 if $P$ holds and to $0$ otherwise.  By Lemma
\ref{minilemma} for any $p,q \in \mathbb{Z}$ we have
\begin{gather}
  \begin{aligned}
    \res_{y=0}\biggl[ \frac{e^{ \lambda y }}{ y^{p} (x+y)^{q}} \biggr] & =
    \res_{y=0} \biggl [\frac{g_{\lambda, -x, q}(y) }{y^p} \biggr ] =
    \frac{[p\geq 1]}{(p-1)!} D^{p-1} g_{\lambda, -x, q} (0) = \\ &
    =\frac{ 1 }{(p-1)!}  \sum_{j=0}^{p-1} \binom{p-1}{j} \frac{
      (-q)\falfa{j}  \lambda^{p-1-j}}{x^{q+j}} = \\
    & =\sum_{j=0}^{p-1} \frac{
      (-q)\falfa{j} \lambda^{p-1-j}}{j! (p-1-j)! x^{q+j}} .
  \end{aligned}
\end{gather}
In the same way
\begin{gather}
  \label{a1b}
  \res_{y=-x}\biggl[ \frac{e^{ \lambda y }}{ y^{p} (x+y)^{q}} \biggr] =
  e^{-\lambda x}\sum_{j=0}^{q-1} \frac{ (-1)^{p+j}(-p)\falfa{j} \lambda^{q-1-j}}{j!
    (q-1-j)! x^{p+j}} .
\end{gather}
Hence
\begin{gather*}
  \label{a2}
  \res_y^+\biggl[ \frac{e^{ \lambda_2 y }}{ y^{b} (x+y)^{c}} \biggr] = \\
  % % [\lambda_2 \geq 0] \cdots \res_{y=0}\biggl[ \frac{e^{ \lambda_2 y }}{ y^{b}
  % % (x+y)^{c}} \biggr] + [\lambda_2 \geq 0] \cd \res_{y=-x}\biggl[
  % % \frac{e^{ \lambda_2 y }}{ y^{b}
  % % (x+y)^{c}} \biggr] = \\
  = [\lambda_2 \geq 0] \cdot \sum_{j=0}^{b-1} \frac{ (-c)\falfa{j}
    \lambda_2^{b-1-j}}{j!  (b-1-j)! x^{c+j}} + [\lambda_2 \geq 0]\cdot
  \sum_{j=0}^{c-1} \frac{ (-1)^{b+j}(-b)\falfa{j} \lambda_2^{c-1-j}
    e^{-\lambda_2 x}}{j! (c-1-j)!  x^{b+j}} .
\end{gather*}
We plug this in \eqref{ameno1}:
\begin{gather}
  \label{a5}
  \begin{gathered}
    R(\lambda_1, \lambda_2, a,b,c) = [\lambda_2 \geq 0]\cdot \sum_{j=0}^{b-1} \frac{
      (-c)\falfa{j} \lambda_2^{b-1-j}}{j! (b-1-j)!} \cdot \res_x^+ \frac{e^{\lambda_1 x}}{x^{a+c+j}} +\\
    + [\lambda_2 \geq 0]\cdot \sum_{j=0}^{c-1} \frac{
      (-1)^{b+j}(-b)\falfa{j} \lambda_2^{c-1-j} }{j! (c-1-j)!}  \cdot
    \res_x^+ \frac{e^{(\lambda_1 - \lambda_2)x} }{x^{a+b+j}} .
  \end{gathered}
\end{gather}
By Lemma \ref {minilemma}
\begin{gather*}
  \res_{x=0} \frac{ e^{\lambda x} } {x^k} = \frac{ [k \geq 1] \cdot
    \lambda^{k-1} } { (k-1)!},\qquad \res^+_x \frac{ e^{\lambda x}} {x^k} =
  [\lambda \geq 0] \cdot [k \geq 1] \cdot \frac{\lambda^{k-1}} { (k-1)!} .
\end{gather*}
So in \eqref{a5} only the terms with $a+b+j \geq 1$ survive:
\begin{gather*}
  \label{a6}
  \begin{gathered}
    R(\lambda_1, \lambda_2, a,b,c,) =
    \\
    = [\lambda_1\geq 0]\cdot [\lambda_2 \geq 0] \cdot
    \sum_{j=\max\{0,1-a-c\}}^{b-1} \frac{ (-c)\falfa{ j}
      \lambda_1^{a+c+j-1} \lambda_2^{b-j-1} }{ j!  (b-j-1)!  (a+c+j-1)!  }
    +\\
    + [\lambda_1 \geq \lambda_2]\cdot [\lambda_2 \geq 0] \cdot
    \sum_{j=\max\{0,1-a-b\}}^{c-1} \frac{ (-1)^{b+j} (-b)\falfa{j}
      \lambda_2^{c-j-1} (\lambda_1 - \lambda_2 )^{a+b+j-1} }{ j!  (c-j-1)!
      (a+b+j-1)!  }.
  \end{gathered}
\end{gather*}
We have
$\{f \in [3]^m: \lambda_1 \geq \lambda_2 \geq 0\} \subset \{f \in [3]^m:\lambda_1
\geq 0, \lambda_2 \geq 0\}$.  Moreover
$\{f:\lambda_1 \geq 0, \lambda_2 \geq 0\} = \{f: \lambda_1 \geq \lambda_2 \geq 0\}
\cup \{f: \lambda_2 \geq \lambda_1 \geq 0\}$ and this union is disjoint for
generic $d$.  So (at least for generic $d$) we have
\begin{gather*}
  \label{a7}
  R %(\lambda_1, \lambda_2, a,b,c,) = \\
  = [ \lambda_1 \geq \lambda_2 \geq 0] \cdot R_1 %(\lambda_1, \lambda_2, a, b,c )
  + [\lambda_2 \geq \lambda_1 \geq 0]\cdot R_2 %(\lambda_1, \lambda_2, a, b,c )
  \\
  \begin{aligned}
    R_1 %(\lambda_1, \lambda_2, a, b,c )
    &:= \sum_{j=\max\{0,1-a-c\}}^{b-1} \frac{ (-c)\falfa{j} \
      \lambda_1^{a+c+j-1} \lambda_2^{b-j-1} }{ j! (b-j-1)! (a+c+j-1)!  }  +
    \\
    &+\sum_{j=\max\{0,1-a-b\}}^{c-1} \frac{ (-1)^{b+j} (-b)\falfa{j}\
      \lambda_2^{c-j-1} (\lambda_1 - \lambda_2 )^{a+b+j-1} }{ j!  (c-j-1)!
      (a+b+j-1)!  }
    \\
    R_2 %(\lambda_1, \lambda_2, a, b,c )
    &:= \sum_{j=\max\{0,1-a-c\}}^{b-1} \frac{ (-c)\falfa{j} \
      \lambda_1^{a+c+j-1} \lambda_2^{b-j-1} }{ j! (b-j-1)! (a+c+j-1)!  }  .
  \end{aligned}
\end{gather*}
By \eqref{aa1}
\begin{gather}
  \label{aa2}
  \res^\Lambda h_f = (-1)^{m_2} R(\lambda_1, \lambda_2, a, b, c)
\end{gather}
where $m_j=m_j(f), \delta_j=\delta_j(f), \delta=\delta(f)$ and
\begin{gather*}
  a =  m_1 +m_2 -2, \quad b=m_2+m_3-2, \quad c= m_1+ m_3 -2 \\
  a+c -1 = m + m_1 -5 \quad a+b -1 = m + m_2 -5 .
\end{gather*}
Moreover $ m - \sum_{j=1}^3 j \, \cdot\, m_j(f) \equiv_2 -2m_3\mod 2$
and
$ \epsilon(f):= (-1)^ { m - \sum_{j=1}^3 j \, \cdot\, m_j(f) } = (-1) ^{m_2}
$.  Notice that $1-a - c = 5 - m -m_1 \leq 5-m \leq 0$. So
$\max\{0,1-a-c\} = 0$ and similarly $\max\{0,1-a-b\} = 0$.  Hence
\begin{gather*}
  R_1 = R_{11} + R_{12}\\
  \begin{aligned}
    R_{11}&:= \sum_{j=0}^{b-1} \frac{ (-c)\falfa{j} \ \lambda_1^{a+c+j-1}
      \lambda_2^{b-j-1} }{ j! (b-j-1)! (a+c+j-1)!  }  \\
    R_{12}&:=\sum_{j=0}^{c-1} \frac{ (-1)^{b+j} (-b)\falfa{j}\
      \lambda_2^{c-j-1} (\lambda_1 - \lambda_2 )^{a+b+j-1} }{ j!  (c-j-1)!
      (a+b+j-1)!  }.
    \\
    R_2 %(\lambda_1, \lambda_2, a, b,c )
    &:= \sum_{j=0}^{b-1} \frac{ (-c)\falfa{j} \ \lambda_1^{a+c+j-1}
      \lambda_2^{b-j-1} }{ j! (b-j-1)! (a+c+j-1)!  }  .
  \end{aligned}
\end{gather*}
Let us start from $R_{11}$.  We wish to expand it in powers of
$\xi_2$ and $\xi_3$ instead of $\xi_1$ and $\xi_2$.  Set for
simplicity $\delta : = a+c -1 $.  Recall that
$\lambda_1 = -(\xi_2 + \xi_3)$, $\lambda_2 = -\xi_3$ and
$(a+c+j-1) + (b-j-1) = b + \delta -1= 2m - 8$.  So
\begin{gather*}
  R_{11} = \sum_{j=0}^{b-1} \binom{-c}{j} \frac{ (\xi_2+\xi_3)^{\delta+j} \xi_3 ^{b-j-1}}{(b-j-1)!(\delta+j)!}  =\\
  = \sum_{j=0}^{b-1} \sum_{i=0}^{\delta+j} \binom{\delta+j}{i}
  \binom{-c}{j} \frac {1} {(b-j-1)!(\delta+j)!}
  \xi_2^i \xi_3^{2n-8-i}.\\
  \binom{\delta+j}{i} \frac {1} {(b-j-1)!(\delta+j)!}  = \frac{1}{i!
    (b+\delta-1-i)!} \binom {b+\delta -1 -i }{\delta+j-i}
  \\
  R_{11} = \frac{1}{(2m-8)!}  \sum_{j=0}^{b-1} \sum_{i=0}^{\delta+j}
  z_{ij}
  \xi_2^i \xi_3^{2n-8-i}  \\
  \text{with } z_{ij} : = \binom{-c}{j} \binom{2m-8}{i} \binom
  {2m-8-i}{\delta +j -i} \quad i,j \in \mathbb{Z}.
\end{gather*}
% \begin{gather*}
%   R_{11} = \sum_{j=0}^{b-1} \binom{-c}{j} \frac{ (\alpha_2+\alpha_3)^{a+c+j-1} \alpha_3 ^{b-j-1}}{(b-j-1)!(a+c+j -1)!}  =\\
%   = \sum_{j=0}^{b-1} \sum_{i=0}^{a+c+j-1} \binom{a+c+j-1}{i}
%   \binom{-c}{j} \frac {1} {(b-j-1)!(a+c+j -1)!}  \alpha_2^i
%   \alpha_3^{2n-8-i}.
% \end{gather*}
We rearrange by summing first on $i$ next on $j$. The result is as
follows:
\begin{gather}
  \label{grunfidi}
  \begin{gathered}
    R_{11} = R_{111} +R_{112}\\
    R_{111} := \frac{1}{(2m-8)!}  \sum_{i=0}^{\delta} \sum_{j=0}^{b-1}
    z_{ij} \xi_2^i
    \xi_3^{2n-8-i} \\
    R_{112} := \frac{1}{(2m-8)!}  \sum_{i=\delta+1}^{2m-8}
    \sum_{j=i-\delta}^{b-1} z_{ij} \xi_2^i \xi_3^{2n-8-i} .
  \end{gathered}
\end{gather}
\begin{lemma}
  $z_{ij}=0$ if $j<0$ or $j>b-1$.  If $j< i-\delta$, then $z_{ij}=0$.
\end{lemma}
\begin{proof}
  If $j<0$, then $\binom{-c}{j}=0$. If $j > b-1$, then
  $\delta+j-i > \delta+ b -1 -i = 2m -8 -i$, so
  $ \binom {2m-8-i}{\delta +j -i} =0$.  Finally if $j< i-\delta$, then
  $\delta+j-i < 0$, so again $ \binom {2m-8-i}{\delta +j -i} =0$.
\end{proof}
It follows that both sums in \eqref{grunfidi} can be extended over
arbitrary integer $j$. Hence
\begin{gather*}
  R_{111} =\frac{1}{(2m-8)!}  \sum_{i=0}^{\delta} \ZZ_i \xi_2^i
  \xi_3^{2n-8-i} \quad R_{112} = \frac{1}{(2m-8)!}
  \sum_{i=\delta+1}^{2m-8} \ZZ_i
  \xi_2^i \xi_3^{2n-8-i} \\
  R_{11} = \frac{1}{(2m-8)!}  \sum_{i=0}^{2m-8} \ZZ_i \xi_2^i
  \xi_3^{2n-8-i} \quad \text{with } \ZZ_i : = \sum_{j} z_{ij}.
\end{gather*}
Using \eqref{tribin}
\begin{gather*}
  \ZZ_i = \binom{2m-8}{i} \sum_j \binom {2m-8-i}{\delta +j -i}
  \binom{-c}{j} = \binom{2m-8}{i}
  \binom{ 2m -8-i -c }{2m -8 - \delta} = \\
  = \binom{2m-8}{i} \binom{ m +m_2 -6 -i}{m_2 +m_3 -3} \\
  R_{11} = \frac{1}{(2m-8)!}  \sum_{i=0}^{2m-8} \binom{2m-8}{i}
  \binom{ m +m_2 -6 -i}{m_2 +m_3 -3} \xi_2^i \xi_3^{2n-8-i}
\end{gather*}
Next we analyse the term $R_{12}$.  Set $\gamma := c-1$, $\epsilon:= a+b
-1$. Notice that $\epsilon \geq 0$ and $\gamma+\epsilon = a+b + c -2 =2m -8$.
Substituting $\lambda_2 = -\xi_3$ and $\lambda_1 - \lambda_2 = -\xi_2$ we get
\begin{gather*}
  R_{12}
  % =\sum_{j=0}^{c-1} \frac{ (-1)^{b+j} (-b)\falfa{j}\
  % \lambda_2^{c-j-1} (\lambda_1 - \lambda_2 )^{a+b+j-1} }{ j!  (c-j-1)!
  % (a+b+j-1)!  } =
  % \\
  =(-1)^b\sum_{j=0}^{\gamma} (-1)^{j} \binom{-b}{j} \frac{1 }{
    (\gamma-j)!  (\epsilon+j)!  } \xi_2 ^{\epsilon+j} \xi_3^{\gamma-j} \\
  % i=\epsilon + j \quad j= i -\epsilon \\
  =\frac{ (-1)^{b+\epsilon}}{(2m-8)!}  \sum_{i=\epsilon}^{2m-8} (-1)^{i}
  \binom{-b}{i-\epsilon} \binom{2m-8}{i} \xi_2 ^{\epsilon+j} \xi_3^{\gamma-j} .
\end{gather*}
Since $\binom{-b}{i-\epsilon} =0$ for $i< \epsilon$, we have indeed
\begin{gather*}
  R_{12}=\frac{ (-1)^{b+\epsilon}}{(2m-8)!}  \sum_{i=0}^{2m-8} (-1)^{i}
  \binom{-b}{i-\epsilon} \binom{2m-8}{i} \xi_2 ^{\epsilon+j} \xi_3^{\gamma-j} .
  \\
  R_1 = R_{11} + R_{12} = \frac{1}{(2m-8)!}  \sum_{i=0}^{2m-8}
  \binom{2m-8}{i}
  \WW_i \xi_2^i \xi_3^{2n-8-i} \\
  \text{with } \WW_i := \binom{ m +m_2 -6 -i}{m_2 +m_3 -3} +
  (-1)^{b+\epsilon+i} \binom{-b}{i-\epsilon} =  \\
  =\binom{\epsilon-i-1}{b-1} +
  (-1)^{b+\epsilon+i} \binom{-b}{i-\epsilon}.
\end{gather*}
To compute $\WW_i$ we distinguish three cases.
\begin{gather*}
  \WW_i =
  \begin{cases}
    (-1)^{b+\epsilon+i} \binom{-b}{i-\epsilon} & \text{if }     b<1,\\
    \binom{\epsilon-i-1}{b-1} & \text{if } i-\epsilon < 0.
  \end{cases}
\end{gather*}
If $b\geq 1$ and $\epsilon \leq i$, then we can use \eqref{gkpp} to get
\begin{gather*}
  (-1)^{i-\epsilon} \binom { -(b-1)-1}{i-\epsilon} = (-1)^{b-1} \binom { - (i-\epsilon) -1}{b-1}\\
  (-1)^{i+\epsilon} \binom { -b}{i-\epsilon} = -(-1)^{b} \binom { \epsilon -i -1}{b-1}\\
  \WW_i = 0.
\end{gather*}
Hence if $b\geq 1$
\begin{gather*}
  R_1 = \frac{1}{(2m-8)!} \sum_{i=0}^{\epsilon-1} \binom{2m-8}{i}
  \binom{\epsilon-i-1}{b-1}  \xi_2^i \xi_3^{2n-8-i} =\\
  = \frac{1}{(2m-8)!} \sum_{i=0}^{m+m_2-6} \binom{2m-8}{i}
  \binom{m+m_2-6-i}{m_2+m_3-3} \xi_2^i \xi_3^{2n-8-i}.
\end{gather*}
If instead $b<1$, then $R_{11}=0$, so
\begin{gather*}
  R_1= R_{12}=\frac{ (-1)^{b+\epsilon}}{(2m-8)!}  \sum_{i=0}^{2m-8}
  (-1)^{i} \binom{-b}{i-\epsilon} \binom{2m-8}{i} \xi_2 ^{\epsilon+j}
  \xi_3^{\gamma-j} .
\end{gather*}
By \eqref{gkp2}
\begin{gather*}
  \binom{-b}{i-\epsilon} = (-1)^{i-\epsilon} \binom {b+i-\epsilon -1}{i-\epsilon},
\end{gather*}
and this vanishes by \eqref{zero} since $b+i-\epsilon -1 < i-\epsilon$.  Hence
for $b<1$, $R_1=0$. But for $b<1 $  $\binom{\epsilon - i -1}{b-1} =0$. So we have at any case
\begin{gather}
  \label{R1}
R_1  = \frac{1}{(2m-8)!} \sum_{i=0}^{m+m_2-6} \binom{2m-8}{i}
  \binom{m+m_2-6-i}{m_2+m_3-3} \xi_2^i \xi_3^{2n-8-i}.
\end{gather}

Finally let us analyse $R_2$.
 Substituting the expressions for  $ \lambda_1$ and $\lambda_2$ from \eqref{lambdazzi} we get
  \begin{gather*}
    R_2 := \sum_{j=0}^{b-1} \frac{ (-c)\falfa{j} \ \xi_1^{a+c+j-1}
      (\xi_1+\xi_2)^{b-j-1} }{ j! (b-j-1)! (a+c+j-1)!  }  = \\
    = \sum _{j=0}^{b-1} \sum_{i=0}^{b-j-1} \binom{b-j-1}{i} \frac{
      (-c)\falfa{j} }{ j! (b-j-1)! (a+c+j-1)!  } \xi_1^{a+c+j-1 + b
      -j -1 -i} \xi_2^i
  \end{gather*}
  % Observe that $ a+c+ b -2 = 2m -8$.
  Set
  \begin{gather*}
    z_{ij} : = \binom{b-j-1}{i} \frac{ (-c)\falfa{j} }{ j! (b-j-1)!
      (a+c+j-1)!  }\qquad \ZZ_i:= \sum_{j=0}^{b-i-1} z_{ij} .
  \end{gather*}
  Then
  \begin{gather*}
    R_2 = \sum_{\substack{i, j \geq 0 \\ i+j \leq b-1}} z_{ij} \,
    \xi_1^{2m-8-i} \xi_2^i = \sum_{i=0}^{b-1} \ZZ_i\,
    \xi_1^{2m-8-i} \xi_2^i.
  \end{gather*}
  Set for simplicity $\beta = b-1-i , \gamma = c-1, \delta = a+ c -1 $.
  Observe that $\beta + \delta = a +b + c -2 -i = 2m -8 -i $.  Then
  \begin{gather*}
    \begin{aligned}
      z_{ij} &= \frac{ (-c)\falfa{j} }{ j! i! (b-j-i-1)!  (a+c+j-1)!
      }  = \binom{-c }{j} \frac{1}{i!} \frac{1}{(\beta -
        j)! (\delta +j)!} =\\
      &= \binom{-c }{j} \frac{1}{i!} \binom{\beta +\delta}{\delta + j}
      \frac{1}{(\beta +\delta)!} = \frac{1}{i! (2m-8-i)!}  \binom{-c
      }{j}\binom{\beta
        +\delta}{\delta + j} = \\
      &= \frac{1}{(2m-8)!} \binom{2m-8}{i} \binom{-c }{j}\binom{\beta
        +\delta}{\delta + j}
      \\
      \ZZ_i &= \frac{1}{(2m-8)!} \binom{2m-8}{i} \sum_{j=0}^\beta
      \binom{\beta +\delta}{\delta + j} \binom{-c }{j}.
    \end{aligned}
  \end{gather*}
  By definition \eqref{fattore}
  $ \binom{-c }{j}\binom{\beta +\delta}{\delta +j} = \binom{-c
  }{j}\binom{\beta +\delta}{\beta-j} = 0$ if $j<0$ or $j > \beta$.
  Hence using \eqref{tribin} (since $\beta+\delta \geq 0$)
  \begin{gather*}
    \sum_{j=0}^\beta \binom{\beta +\delta}{\delta + j} \binom{-c }{j}
    = \sum_j \binom{\beta +\delta}{\delta + j} \binom{-c }{j} =
    \binom{\beta+\delta -c} {\beta}
  \end{gather*}
  Now $\beta +\delta -c = m +m_2 -6 -i$ and
  $ \delta - c = a-1 = m_1 + m_2 - 3$. So
  \begin{gather*}
    \sum_{j=0}^\beta \binom{\beta +\delta}{\delta + j} \binom{-c }{j}
    =
    \binom{ m+m_2 -6 -i }{ m_1+m_2 -3} \\
    \ZZ_i = \frac{1}{(2m-8)!} \binom{2m-8}{i } \binom{ m+m_2 -6 -i }{
      m_1+m_2 -3} .
  \end{gather*}
We get  
  \begin{gather}
    \label{R2}
    R_2 = \frac{1}{(2m-8)!}  \sum_{i=0}^{b-1} \binom{2m-8}{i } \binom{
      m+m_2 -6 -i }{ m_1+m_2 -3} \xi_1^{2m-8-i} \xi_2^i.
  \end{gather}
  To conclude
we apply the residue theorem, \eqref{aa2}, \eqref{R1} and \eqref{R2}
  \begin{gather*}
    \mathrm{Vol}(M_d)  = \\
    = \frac{n_0 (-1)^{s+n_+}} {|W|} \sum_{f\in [3]^m }
    \res^\Lambda h_f(X) =\\
    = -\frac{1}{6} \sum_{f\in [3]^m } (-1)^{m_2} \Bigl ( [\lambda_1 >
    \lambda _2 > 0]\cdot R_1 +
    [\lambda_2 > \lambda _1 > 0]\cdot R_2  \Bigr ) = \\
    = -\frac{1}{6} \sum_{\mathcal{A}(d)} (-1)^{m_2} R_1 +
    -\frac{1}{6} \sum_{\mathcal{B}(d)} (-1)^{m_2} R_2 =\\
    = -\sum_{f \in \mathcal{A}(d)} \frac{ (-1)^{m_2(f)} }{6(2m-8)!}
    \sum_{j=0}^{2m-8} \binom{2m-8}{j}\binom{m+m_2-6-j}{m_2 + m_3 -3} \xi_2^j \xi_3^{2m-8 -j} +\\
    - \sum_{f \in \mathcal{B}(d)} \frac{ (-1)^{m_2(f)} }{6(2m-8)!}
    \sum_{j=0}^{2m-8} \binom{2m-8}{j}\binom{m+m_2 -6-j}{m_1 + m_2 - 3}
    \xi_1^j \xi_2^{2m-8-j}.
      \end{gather*}
  This completes the proof of the Theorem. \hfill $\Box$

\end{dimostrazione}

\begin{osservazione}
This formula was found  in 2008 by Suzuki and Takakura \cite{SuTa08} with a different technique. To restore exactly their result of  \cite[Theorem 5.6]{SuTa08}, one  needs to interchange the index $1$ with the index $3$. However, our method is more general and in principle can be applied (with longer computations) to any dimension.
\end{osservazione}

%%%%%%%%%%%%%%%% fino a qui    

\section{Final Remarks}

The main results of this work relies on the study of the geometric properties when the $K$-moduli is proper, projective and explicit. \\ 
In \cite{OSS} it is proven that the $K$-moduli space of del Pezzo of degree $d \in \{1,2,3,4\}$ is proper and projective, and moreover when $d=3,4$ there is a nice GIT description of such moduli spaces. In degree $3$ we have the moduli space of cubics

$$ M_3 = \mathbb{P}(\mathrm{Sym}^{3}(\mathbb{C}^4)^{*}) // \mathrm{SL}(4). $$

In this case, the GIT-stability conditions tells that the semistable locus do not match with the stable locus. Then, in this case the Kirwan map is no longer surjective and therefore the Jeffrey-Kirwan non abelian localization can not be applied. One way to proceed is considering the \emph{Kirwan partial desingularization} \cite{Kir85}, and calculate the volume of the desingularized GIT quotient. It may be interesting to compute the volume also in such case. \\
In Definition \ref{logwp} we have defined the log Weil-Petersson metric via fiber integral. It would be interesting to restore the classical definition of the log Weil-Petersson metric. That is, define the Weil-Petersson metric as an $L^2$-metric via the Kodaira-Spencer map. For general type geometries, Schumacher and Trapani in \cite{ST13} gave a decription of the classical definition of Weil-Petersson metric.

\end{document}